\pdfoutput=1
\documentclass[10pt,reqno]{amsart}

\usepackage[utf8]{inputenc}
\usepackage{geometry}
\usepackage{amsthm, amsmath, amssymb, amscd, amsfonts, environ}
\usepackage{mathtools,stackengine}
\usepackage{setspace}
\usepackage{enumitem}
\usepackage{subfig}
\usepackage{tikz}
    \usetikzlibrary{arrows,shapes,trees,cd}
    \usetikzlibrary{decorations.markings}
    \usetikzlibrary{decorations.pathmorphing}
\usepackage{faktor}
\usepackage{wrapfig}
\usepackage{hyperref}
\usepackage[capitalize]{cleveref}
\usepackage{multicol}
\usepackage{adjustbox}

\title{Symmetry of $f$-vectors of toric arrangements in general position and some applications}
\author{D. Bergerová}
\date{\today}
\keywords{Hyperplane arrangements, toric arrangements, deletion-restriction}

\theoremstyle{plain}% default
\newtheorem{thm}{Theorem}[section]
\newtheorem{lem}[thm]{Lemma}
\newtheorem{prop}[thm]{Proposition}
\newtheorem{cor}[thm]{Corollary}
\newtheorem*{unlabeledthm}{Theorem}

\theoremstyle{definition}
\newtheorem{defn}{Definition}[section]
\newtheorem{exmp}{Example}[section]

\theoremstyle{remark}
\newtheorem*{rem}{Remark}

\newcommand\ledot{\mathrel{\ensurestackMath{%
  \stackengine{-.5ex}{\lessdot}{-}{U}{c}{F}{F}{S}}}}

\AtBeginDocument{} % Big Oh
\AtBeginDocument{} % Complex
\newcommand{\R}{\mathbb{R}} % Reals
\newcommand{\Z}{\mathbb{Z}} % Integers
\newcommand{\N}{\mathbb{N}} % Naturals
\newcommand{\T}{\mathbb{T}} % Torus
\newcommand{\A}{\mathcal{A}} % Arrangement A
\newcommand{\M}{\mathcal{M}} % moduli space
\newcommand{\I}{\mathcal{I}}

\renewcommand{\Vec}{\mathbf}
\renewcommand{\phi}{\varphi}

\DeclareMathOperator{\codim}{\mathrm{codim}}

\DeclareMathOperator{\Int}{\mathrm{Int}}

\begin{document}
\singlespacing

\maketitle

\begin{abstract}
    A toric hyperplane is the preimage of a point $x \in S^1$ of a continuous surjective group homomorphism $\theta: \T^n \to S^1$. A finite hyperplane arrangement is a finite collection of such hyperplanes. In this paper, we study the combinatorial properties of finite hyperplane arrangements on $\T^n$ which are spanning and in general position. Specifically, we describe the symmetry of $f$-vectors arising in such arrangements and a few applications of the result to count configurations of hyperplanes.
\end{abstract}

%%%%%%%%%%%%%%%%%%%%%%%%%%%%%%%% new section %%%%%%%%%%%%%%%%%%%%%%%%%%%%%%%%%%%%%%%%%%%%%%
\section{Introduction}
\subsection{Hyperplane arrangements}
Hyperplane arrangements in the setting of $\R^n$ have a rich history, with \cite{orlik1992arrangements} serving as a standard reference. In recent years, an alternative setting of $\T^n$ has been studied more prominently, from combinatorial, algebraic, and topological point of views; see \cite{Ehrenborg_2009,corradodeconcini_procesi_2005} for examples. Some applications of the theory of toric arrangements to toric varieties have been presented in \cite{hicks_2021,hanlon2023resolutions}. 

In this paper, we study finite hyperplane arrangements on $\T^n$ predominantly from the combinatorial perspective. We take the foundation for real hyperplane arrangements laid out in \cite{stanley_2015} and adapt it to the toric case. The motivation of the paper is to enumerate the toric arrangement of hyperplanes with fixed normal vectors under an equivalence relation. A natural topological invariant to study with hyperplane arrangements are the flats arising from arrangements. The enumerative information of $k$-flats can be encoded in an \emph{$f$-vector}. We prove the following relation between the flats of arrangements in general position.

\begin{unlabeledthm}[\cref{prop:facetsduality}]
    For a spanning toric arrangement $(\A,\T^n)$ in general position, the $f$-vector $\Vec{f}(\A)$ is symmetric.
\end{unlabeledthm}

In other words, the number of $k$-flats in a reasonable toric arrangement $\A$ is equal to the number of $(n-k)$-flats of $\A$.

We also present a few applications of the theorem above. We say that two toric arrangements $\A_1$ and $\A_2$ of the same set of hyperplanes are \emph{equivalent}, $\A_1 \sim \A_2$, if we can continuously translate a subset of hyperplanes in $\A_1$ to $\A_2$ without ever creating a triple intersection. 

\begin{unlabeledthm}[\cref{thm:upperboundonarrangements}, \cref{lem:determinatgivesregions}, \cref{cor:regionsofspanningarrangements}]
    Let $\A = \{H_1,H_2,\ldots,H_m\}$ be a set of $2$-dimensional toric hyperplanes given by normal vectors $\Vec{a_1},\Vec{a_2},\ldots,\Vec{a_m}$. Let $\M$ denote the space of all distinct arrangements $(\A,\T^n)$. Then we can construct a space $(\M',\T^{m-2}) \hookleftarrow \M/\sim$ which parametrizes translation of the hyperplanes. The $(m-2)$-flats of the parameter space $\M'$ correspond to arrangements of the $m$ lines in general position. We provide an upper bound on the number of arrangements later.
\end{unlabeledthm}

\subsection{Outline} 
The paper is structured as follows. After introduction, we list the main results of this paper on a toy example of toric hyperplane arrangements on $\T^2$ in \cref{sec:sectiontwo}. We then introduce real hyperplane arrangements and adjust the definitions for toric hyperplane arrangements in \cref{sec:sectionthree}. Next, in \cref{sec:sectionfour}, we also state and generalise the deletion-restriction relation to $n$-dimensional tori and $k$-dimensional flats for arrangements in general position in \cref{lem:deletionrestriction}. The generalised deletion-restriction then allows us to show \cref{prop:facetsduality}, stating that the number of $k$-flats is equal to the number of $(n-k)$-flats on arrangements on $\T^n$, for which we list a few applications. One of the applications of the theorem is counting the connected components of the complement of a toric hyperplane arrangement as shown in \cref{sec:sectionfive}. We end the paper with \cref{sec:sectionsix}, where we focus on classifying hyperplane arrangements on $\T^2$ under the following equivalence relation: we will say that two arrangements of the same set of hyperplanes are equivalent if and only if we can continuously translate a subset of hyperplanes without creating a triple intersection of one arrangement to the other arrangement. We show a few examples of constructing geometric spaces classifying hyperplane arrangements on $\T^2$ and give an upper bound for the number of arrangements for a set of hyperplanes with fixed normal vectors.

\subsection*{Acknowledgment}
I would like to thank Jeff Hicks for suggesting the topic and supervising the undergraduate project. This paper also benefitted greatly from Sophie Bleau's comments on early drafts. This work was partially supported by Simons Investigator Award No. 929034 through Nick Sheridan.

\section{The two-dimensional case}
\label{sec:sectiontwo}

To concretely present the results, we begin by looking at the case of $\T^2$ which provides intuition for the general case. We sometimes refer to hyperplanes as \emph{lines} when working on $\T^2$. We will provide proofs to general results later.

\subsection{Symmetry of \texorpdfstring{$f$}{Symmetry of f-vectors}-vectors}

Consider the arrangement $(\A,\T^2)$ of hyperplanes $H_1$ and $H_2$ given by the normal vectors $\langle 1,2 \rangle$ and $\langle 1,-2 \rangle$, respectively, shown in \cref{fig:firstexampleoflinearrangement}. The arrangement $\A$ comes with vertices, edges and regions. Denote the number of vertices by $f_0$, the number of edges by $f_1$, and the number of regions by $f_2$. Then the $f$-vector $\Vec{f}(\A)$ for $\A$ is the vector with the $i$-th entry being $f_i$. In our example, we have $\Vec{f}(\A) = (4,8,4)$, in particular, $f_0 = f_2$ for the arrangement $\A$. In fact, for a “reasonable” (in general position, more precisely) line arrangement, $\A$ on $\T^2$, the number of regions $f_2$ is equal to the number of vertices $f_0$.

\begin{figure}[httt]
    \centering
    \begin{tikzpicture}[scale=.7]
        \tikzstyle{invisible} = [outer sep=0,inner sep=0,minimum size=0]
    \tikzstyle{int} = [circle,fill,outer sep=2,inner sep=2,minimum size=0]

        \begin{scope}[thick,decoration={
        markings,
        mark=at position 0.5 with {\arrow{>}}}
        ] 
        \draw[postaction={decorate}] (-4,0)--(2,0);
        \draw[postaction={decorate}] (2,6)--(2,0) node[invisible] (v3) {};
        \draw[postaction={decorate}] (-4,6)--(2,6) node[invisible] (v2) {};
        \draw[postaction={decorate}] (-4,6) node[invisible] (v6) {}--(-4,0) node[invisible] (v1) {};
    \end{scope}

    \draw[blue] (v1) -- (2,3) node[invisible] (v5) {};
    \draw[blue] (-4,3) node[invisible] (v4) {} -- (v2);
    \draw[orange] (v3) -- (v4);
    \draw[orange] (v5) -- (v6);
    \node[int] at (-1,1.5) {};
    \node[int] at (-1,4.5) {};
    \node[int] at (-4,0) {};
    \node[int] at (-4,3) {};
    
    \node at (-3.0997,5.1261) {$H_1$};
    \node at (0.6736,1.9513) {$H_2$};
    \node at (-3.6981,0.5313) {$V_1$};
    \node at (-1,2) {$V_2$};
    \node at (-3.12,3.0164) {$V_3$};
    \node at (-1,5) {$V_4$};
    \draw (-0.9789,6.0718) -- (-0.9096,5.997) -- (-0.9964,5.909);
    \draw (-0.9969,0.0809) -- (-0.9306,0.0076) -- (-1.0076,-0.075);
    \end{tikzpicture}
    \caption{A toric arrangement on $\T^2$ given by hyperplanes $H_1 = \langle 1,2 \rangle$ and $H_2 = \langle 1,-2\rangle$.}
    \label{fig:firstexampleoflinearrangement}
\end{figure}
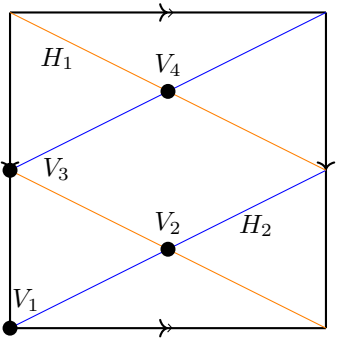

The claim can be proven using basic graph theory. For any arrangement of lines $\A$, let $f_0(\A)$ be the number of vertices in $\A$ and $f_1(\A)$ be the number of edges in $\A$. By a standard graph-theoretic result, $$\sum _{i = 1} ^n \deg (v_i) = 2f_1(\A).$$ Since $\A$ is in general position, any vertex is an intersection of exactly two lines, so the degree of any vertex is 4. Thus, we have $4f_0(\A) = 2f_1(\A)$, implying $2f_0(\A) = f_1(\A)$. Since the Euler characteristic of $\T^2$ is given by $f_0(\A) - f_1(\A) + f_2(\A) = 0$, substituting our previous result yields $f_0(\A) = f_2(\A)$ as required. However, such reasoning explicitly relies on the Euler characteristic, which is only applicable to arrangements with contractible faces. Hence, this reasoning is not applicable to cases with non-simply connected regions, like arrangements with all lines parallel. We also need to assume, for this argument, that our arrangements are reasonably nice in a sense that not all hyperplanes are parallel and there are no triple intersections.

\subsection{Counting regions of arrangements}
The symmetry of $f$-vectors can be utilised to count the number of regions of an arrangement $\A$ from the normal vectors of hyperplanes. Let $\Vec{a_i}$ denote the normal vector of a line $H_i$. A standard result in algebraic topology states that the number of intersections arising between two lines $H_1$ and $H_2$ is the absolute value of the determinant
% For the purpose of the next proof, we will use the notion of signed intersections. First, we need to fix the \emph{orientation} of each line in \(\mathcal{A}\). For example, we have \(\langle 1,0 \rangle \cap \langle 0,1 \rangle = +1\) and \(\langle 0,1 \rangle \cap \langle 1,0 \rangle = -1\). The following is a standard result in algebraic topology.
% Let \(C_3\) be any curve in \(\mathbb{R} ^2\) avoiding \(C_1 \cap C_2\). Then,
% \[ |(C_1 \# C_2) \cap C_3 
%  = | C_1 \cap C_3 | + | C_2 \cap C_3 |, \]
% where \(C_1 \# C_2\) stands for an orientation-preserving “addition" of curves. We can assume that \(L_1 \# L_2 = \langle a + c, b + d\rangle\) and that the signs of intersections are invariant under this operation.
\begin{align*}
    f_0(\A) = |\det (\Vec{a_1}| \Vec{a_2})|.
\end{align*}

Generalising to $m > n$ lines, and assuming the arrangement is still in general position, we have that the number of intersections is given by
\[ f_0(\A) = \sum _{H_i, H_j \in \mathcal{A}} |\det( \Vec{a_i} | \Vec{a_j})| . \]
When our arrangement is in general position, we use the symmetry $f_0(\A) = f_2(\A)$, so we only need to count the number of line intersections. Using the topological fact that $f_0(\A) - f_1(\A) + f_2(\A) = 0$, we also obtain $f_1(\A) = 2f_0(\A)$. Therefore, we have reduced the problem of counting the regions $f_2(\A)$ of toric arrangements to counting intersections between the hyperplanes.

\subsection{Counting arrangements under equivalence}

We take two line arrangements to be equivalent if and only if we can continuously translate a subset of lines in one arrangement to get the other arrangement without creating a triple intersection at any point. Denote the equivalence of two arrangements by $\sim$. The question is: how many distinct line arrangements are there given a set of lines with fixed slope? 

Suppose that we have an arrangement $\A^\prime$ on $\T^2$ such that $|\A^\prime| = 2$. There is only one such arrangement, as there is no possibility of a triple intersection when translating either of the hyperplanes. 

We give an intuitive way of counting the maximum number of distinct arrangements of 3 lines. We consider adding a new hyperplane $H_3$ to $\A^\prime$. Since we want to enumerate distinct arrangements subject to $\sim$, we look for arrangements containing triple intersections. We are able to give an upper bound for the possible triple intersections by looking at complete (i.e. double) intersections of the original arrangement $\A^\prime$. Through each complete intersection, we place the new hyperplane $H_3$, so there will be at most $f_0(\A)$ copies of $H_3$. Note that $H_3$ may pass through more than one complete intersection, which disregards the symmetry of $\T^2$, hence we only get an upper bound. However, to obtain a better bound, we need to consider this process with every hyperplane in $\A$. Using the global translation of the torus, we may fix an intersection of a pair of hyperplanes $H_i$ and $H_j$ at the origin. Let $\A_{H_i}$ denote the \emph{deletion} of a line $H_i$ from $\A$, symbolically $\A_{H_i}= \A \setminus H_i$. Therefore, we find that the number of distinct toric arrangements of the three hyperplanes is at most $$\min_{H_i \in \A} (f_0(\A' _{H_i})).$$
Thus, we have established an upper bound for the number of distinct toric arrangements of three lines with fixed normal vectors on $\T^2$. However, we will later see that there is only one arrangement of three lines so that $\A$ is in general position, once we factor in the symmetries of $\T^2$. Hence, the bound is not the tightest, but we address this in the last section.

\section{Background}
\label{sec:sectionthree}

\begin{defn}
By an \emph{\(n\)-torus} $\T^n$, we understand the topological group \(\mathbb{R}^n/\mathbb{Z}^n\).
\end{defn}

There are plenty of ways we can think about tori, and we will use them interchangeably throughout the paper. For example, the $n$-dimensional torus $\T^n$ can be represented by an $n$-dimensional hypercube with the opposing faces identified. We also have $$\T^n = \underbrace{S^1 \times S^1 \times \ldots S^1}_{n}.$$ Another way to look at $\T^n$ is to consider the quotient map $$\pi: \mathbb{R}^n \to \faktor{\mathbb{R}^n}{\sim} \cong \mathbb{T}^n,$$ where $(u_1, u_2, \ldots, u_n) \sim (v_1, v_2,\ldots, v_n)$ if and only if for all $i \in \{1,2,\ldots, n\}$ we have $u_i - v_i \in \Z$, i.e. $(u_1 - v_1, u_2 - v_2, \ldots, u_n - v_n) \in \Z^n$.

Considering that $\R^n$ is a covering space of $\T^n$, many structures in the affine space $\R^n$ have analogues in the torus $\T^n$. For example, we adapt the classical definition of a linear hyperplane to $\T^n$.

\begin{defn}
\label{defn:hyperplaneontorus}
Let \(\theta : \mathbb{T}^n \to S^{1}\) be a surjective group homomorphism. Then we may define a \emph{toric hyperplane} $H$ on $\T^n$ as the kernel of $\theta$, $\ker (\theta)$.
\end{defn}

It follows by the isomorphism theorems that because $\T^n/\ker(\theta) \cong S^1$, we have $\dim(\T^n/\ker(\theta)) = \dim(S^1)$ which subsequently implies $\dim(\ker(\theta)) = n-1$. By analogy to the setting of $\R^n$ a real hyperplane is the kernel of a surjective map $\tilde\theta:\R^n \to \R$,and we have by rank-nullity that $\ker(\tilde\theta) = n-1$. Hence, the definition of a toric hyperplane is consistent with the classical real case.

Fix a basis of $\R^n$. Then any surjective continuous group homomorphism $\theta: \T^n \to S^1$ can be interpreted as a map 
\begin{align*}
    \Vec{x} \mapsto \langle a_1, a_2, \ldots, a_n \rangle\cdot \Vec{x} = a_1x_1 + a_2x_2 + \ldots + a_nx_n,
\end{align*}

where $a_1, a_2, \ldots, a_n$ are integer scalars and at least one $a_i$ is non-zero. Following \cref{defn:affinehyperplaneontorus}, a toric hyperplane $H$ is given by a set of all vectors in $\T^n$ such that $a_1x_1 + a_2x_2 + \ldots + a_nx_n = b$ for a fixed arbitrary constant $b \in \R$. Hence, we can characterise hyperplanes by their normal vector $\langle a_1, a_2, \ldots, a_n \rangle = \Vec{n} \in \Z^n$ and their intercept $b$.

\begin{defn}
\label{defn:affinehyperplaneontorus}
     Let \(\theta:\mathbb{T}^n \to S^{1}\) be a non-zero surjective group homomorphism. Then the preimage $\theta^{-1}(x)$ of a point $x \in S^1$ defines an \emph{affine toric hyperplane $H$} on $\T^n$. We call $x \in S^1$ the \emph{intercept} of $H$.
\end{defn}

We will refer to affine toric hyperplanes as hyperplanes. A toric hyperplane is inherently a special case of an affine toric hyperplane. Furthermore, the definition of an affine toric hyperplane does not pose any restrictions on the map $\theta$, so it is possible that some hyperplanes may appear with multiplicities when we choose the normal vector entries such that $\gcd(a_1,a_2,\ldots,a_n) > 1$.

\begin{exmp}
Let $(\A,\T^1)$ be an arrangement of one hyperplane $H$ given by the normal vector $\Vec{n} = \langle 5 \rangle$, as shown in \cref{fig:nonsimplehyperplanearrangementexample}. The normal vector $\Vec{n}$ can also be expressed as $5\cdot\langle 1 \rangle$ and as a set $\A$ consists of 5 equally spaced copies of $\langle 1 \rangle$ on $\T^1$.
    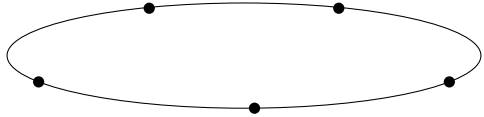
\begin{figure}[ht]
        \centering
        \begin{tikzpicture}[scale=.7]
          \tikzstyle{pt} = [circle,fill,outer sep=1.5,inner sep=1.5,minimum size=1.5]
        \draw  (-4,2.5) ellipse (4.5 and 1);
        \node[pt] at (-7.9,2) {};
        \node[pt] at (-3.8,1.5) {};
        \node[pt] at (-0.1,2) {};
        \node[pt] at (-2.2,3.4) {};
        \node[pt] at (-5.8,3.4) {};
        \end{tikzpicture}
        \caption{An example of an arrangement on $\T^1$ containing a non-simple hyperplane.}
        \label{fig:nonsimplehyperplanearrangementexample}
    \end{figure}
\end{exmp}

This is unlike real hyperplane arrangements for which non-primitive normal vectors just scale the hyperplane without disconnecting it.

\begin{defn}
    Let $H$ be a hyperplane on $\T^n$ defined by normal vectors $\langle a_1, a_2, \ldots, a_n \rangle$. Then we say that $H$ is a \emph{simple hyperplane} if and only if $\gcd(a_1,a_2, \ldots,a_n) = 1$.
\end{defn}

Due to the following proposition, we will assume all hyperplanes to be simple unless otherwise stated.

\begin{prop}
\label{lem:simplehyperplanesareconnected}
    Let $H$ denote a hyperplane on $\T^n$. Then $H$ is connected if and only if $H$ is simple.
\end{prop}

\begin{proof}
    Suppose $H$ is not simple, so we assume $\gcd(a_1,a_2,\ldots,a_n) = m$ with $m > 1$. The hyperplane $H$ is defined as the kernel of the map $\theta: \T^n \to S^1;\; \Vec{x} \mapsto a_1x_1 + a_2x_2 + \ldots + a_nx_n$. We may rewrite the image of $\Vec{x}$ under $\theta$ as $m(b_1x_1 + b_2x_2 + \ldots + b_nx_n).$ Therefore we can consider maps $\theta^\prime: \T^n \to S^1$ given by $\Vec{x} \mapsto b_1x_1 + b_2x_2 + \ldots + b_nx_n$, and a map $\theta_m: S^1 \to S^1$, multiplication by $m$, sending $\Vec{x}$ to $m\Vec{x}$. We factor $\theta$ as in the following commutative diagram.
\[\begin{tikzcd}
    {\T^n} && {S^1} \\
    & {S^1}
    \arrow["\theta", from=1-1, to=1-3]
    \arrow["{\theta^\prime}"', from=1-1, to=2-2]
    \arrow["{\theta_m}"', from=2-2, to=1-3]
\end{tikzcd}\]
    Then $H = \theta^{\prime -1}(\theta_m ^{-1}(c))$, for an intercept $c \in S^1$. Since $\theta_m ^{-1}(c)$ is disconnected, it follows that $H$ is disconnected as well.
    
    Conversely, suppose $H$ is simple where $H$ is given by the kernel of the map $$\theta: \T^n \to S^1; \Vec{x} \mapsto \langle a_1, a_2, \ldots, a_n \rangle \cdot \Vec{x}$$ where $\gcd(a_1,a_2,\ldots,a_n) = 1$. Let $\Vec{v} \in H$ be a vector such that $\langle a_1, a_2, \ldots, a_n \rangle \cdot \Vec{v} = 1$, which exists by the Bézout identity. We may take two points of $H$, say $$\Vec{x} = (x_1,x_2,\ldots,x_n), \quad \text{and} \quad \Vec{y} = (y_1,y_2,\ldots,y_n).$$ By definition of $H$, we have $\theta(\Vec{x}) = 0$ and $\theta(\Vec{y}) = 0$. Denote by $\pi$ the covering space $\pi: \R^n \to \T^n$ and let $\pi^{-1}(\Vec{x}), \pi^{-1}(\Vec{y}) \in \R^n$ be the lifts of $\Vec{x}$ and $\Vec{y}$ respectively.

\[\begin{tikzcd}
    {\R^n} && {\R} \\
    \\
    {\T^n} && {S^1}
    \arrow["\theta", from=3-1, to=3-3]
    \arrow["\tilde\theta", from=1-1, to=1-3]
    \arrow["\pi"', from=1-1, to=3-1]
    \arrow["\pi'", from=1-3, to=3-3]
\end{tikzcd}\]
     The diagram above commutes. Fix specific points $\tilde{\Vec{x}} \in \pi^{-1}(\Vec{x}), \tilde{\Vec{y}} \in \pi^{-1}(\Vec{y})$. To show that $H$ is path-connected, and therefore connected, we will show that there exists a choice of a translation to $\Tilde{\Vec{y}} \in \pi^{-1}(\Vec{y})$ and a path $\gamma$ from $\Tilde{\Vec{x}}$ to $\Tilde{\Vec{y}}$ such that $\gamma \subseteq \pi^{-1}(H)$, where $\pi^{-1}(H)$ is a collection of affine hyperplanes in $\R^n$. Let $\tilde{H}$ be a hyperplane in $\R^n$ so that $\pi(\tilde{H}) = H$, and we fix points $\Tilde{\Vec{x}}$ and $\Tilde{\Vec{y_0}}$. We will call $\tilde{H}$ a representative of $\pi^{-1}(H)$. We try a path given by 
     \[\gamma_k(t): [0,1] \to \R^n; \; t \mapsto (1-t)\tilde{\Vec{x}} + t (\tilde{\Vec{y_0}} + k \Vec{v}), \quad \text{for some $k \in \Z$}.\]
     We want to find a choice of an integer $k$ so that $$\tilde{\theta}(\tilde{\Vec{y_0}} + k \Vec{v}) - \tilde{\theta}(\tilde{\Vec{x}}) = 0$$ which will subsequently imply that $\gamma_k \subseteq \tilde{H}$. 
     By assumption, we have $\tilde{\theta}(\Vec{v}) = 1$, so by linearity, $\tilde{\theta}(\tilde{\Vec{y_0}} + k \Vec{v})$ hits every connected component of $\pi^{-1}(H)$, one of which must satisfy the relation $\tilde{\theta}(\tilde{\Vec{y_0}} + k \Vec{v})) - \tilde{\theta}(\tilde{\Vec{x}}) = 0$. Therefore, we have found a unique representative $\Tilde{\Vec{y}} = \Tilde{\Vec{y_0}} + k\Vec{v}$ in $\pi^{-1}(\Vec{y})$. Now suppose that our choice of $k$ satisfies $\tilde{\theta}(\Tilde{\Vec{x}}) - \tilde{\theta}(\Tilde{\Vec{y_0}} + k\Vec{v}) = 0$. It is left to check that $\pi(\gamma_k) \subseteq H$. Because $\pi^{-1}(H)$ can be thought of as the union of integer translates of $\tilde{H}$ in $\R^n$, we have that $\tilde{\theta}(\gamma_k(t)) = \tilde{\theta}(\tilde{\Vec{x}}) \in \Z$. So, $\pi'(\tilde{\theta}(\gamma_k)) = \pi'(\Z) = 0$. This is equivalent to $\theta(\pi(\gamma_k)) = 0$, which implies $\pi(\gamma_k) \in \ker(\theta)$, hence $\pi(\gamma_k) \subseteq H$. 
\end{proof}
We illustrate the idea in the proof above for $\R^2$ in \cref{fig:exampleforpfoflemma21}.
\begin{figure}[h!]
    \centering
\begin{tikzpicture}[scale=.9]
      \tikzstyle{invisible} = [outer sep=0,inner sep=0,minimum size=0]
        \draw [color=gray!20, step=1cm,very thin] (-4.5,-1) grid (5,6.5);
        \draw[gray!0]  (-4.5,6.5) node[invisible] (v1) {} rectangle (5,-1) node[invisible] (v2) {};
        \draw [-latex](-3,-1) node[invisible] (v4) {} -- (-3,6.5);
        \draw [-latex](-4.5,1) -- (5,1) node[invisible] (v3) {};
        \draw (-4.5,2.5) -- (-1,-1);
        \draw (-4.5,4.5) -- (1,-1);
        \draw (3,-1) -- (v1);
        \draw (v2) -- (-2.5,6.5);
        \draw (v3) -- (-0.5,6.5);
        \draw (1.5,6.5) -- (5,3);
        \draw (3.5,6.5) -- (5,5);
        \draw (v4) -- (-4.5,0.5);
        \node[circle,fill,scale=.3] (v6) at (-1,5) {};
        \node[circle,fill,scale=.3] (v5) at (-2,0) {};
        \node at (-0.8,5.3) {$\tilde{\mathbf{x}}$};
        \node at (-2.2,-0.3) {$\tilde{\mathbf{y_0}}$};
        \draw[red,-latex](v5) -- (-1,1);
        \node at (-1.6,0.7) {$\mathbf{v}$};
        \node[circle,fill,scale=.3] at (-1,1) {};
        \node[circle,fill,scale=.3] at (0,2) {};
        \node[circle,fill,scale=.3] (v7) at (1,3) {};
        \node at (-0.8,1.5) {$\tilde{\mathbf{y_0}} + \mathbf{v}$};
        \node at (0.5,2.3) {$\tilde{\mathbf{y_0}} + 2\mathbf{v}$};
        \node at (1.7,3.1) {$\tilde{\mathbf{y_0}} + 3\mathbf{v}$};
        \draw [-latex,blue] (v6) edge (v7);
        \node at (0.4,4.1) {$\gamma_3$};
    \end{tikzpicture}
    \caption{The lift of a hyperplane $H$ of $\T^2$ into $\R^2$.}
    \label{fig:exampleforpfoflemma21}
\end{figure}
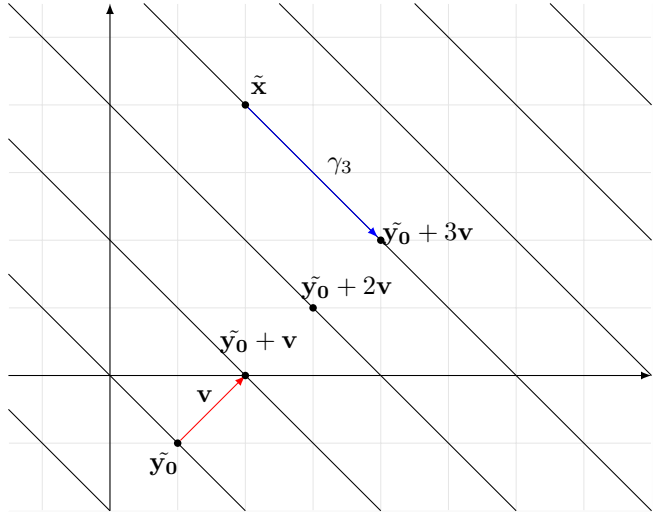

In the classical theory of hyperplane arrangements, it is possible to define a hyperplane arrangement $(\A,V)$ as a set of hyperplanes $\A$ in an ambient $k$-vector space $V$. Although $\T^n$ is not a vector space, we rely on the fact that the covering space $\R^n$ is a vector space throughout the paper. We now define a toric hyperplane arrangement.
\begin{defn}
    A \emph{toric hyperplane arrangement} $(\A,\T^n)$ is a finite set $\A$ of affine toric hyperplanes in the ambient space $\T^n$.
\end{defn}

By a standard abuse of notation, we refer to $\A$ as a toric hyperplane arrangement or simply hyperplane arrangement when the dimension of the ambient space is clear or does not need to be specified. We will also denote the number of hyperplanes in $\A$ by $|\A|$.

Let an arrangement $(\A,\T^n)$ consist of hyperplanes given by maps $\theta_1,\theta_2,\ldots, \theta_m$ and intercepts $b_1,b_2\ldots,b_m$. After fixing a basis, denote by $\Theta$ the \emph{system of hyperplanes of $\A$} which we define to be a matrix in $\mathrm{Mat}_{m\times n}(\Z)$. We can encode the information about hyperplanes in $\A$ using a pair of matrices,
\begin{align*}
    (\Theta, \Vec{b}) = \left(
    \begin{pmatrix}
    \leftarrow \Vec{a_1} \rightarrow \\
    \leftarrow \Vec{a_2} \rightarrow \\
    \vdots \\
    \leftarrow \Vec{a_m} \rightarrow
    \end{pmatrix}, \begin{pmatrix}
        b_1\\b_2\\ \vdots \\ b_m
    \end{pmatrix}\right).
\end{align*}

\begin{defn}[\cite{stanley_2015}]
An arrangement \( (\A, \T^n) \) is in general position if 
\begin{align*}
    & \bullet \{ H_1, H_2, \ldots, H_m \} \subseteq \mathcal{A}, m \leq n \colon \quad \dim (H_1 \cap H_2 \cap \ldots \cap H_m) \leq
    n - m,\\
    & \bullet \{ H_1, H_2, \ldots, H_m \} \subseteq \mathcal{A}, m > n \colon \quad H_1 \cap H_2 \cap \ldots \cap H_m = \emptyset .
\end{align*} 

\end{defn}
    
In the case of arrangements on $\T^2$, a hyperplane arrangement in general position roughly means arrangements without triple intersections. Unless otherwise stated, every arrangement will be assumed to be in a general position. We also wish to exclude the analogue of arrangements with all hyperplanes parallel in the case of $\T^2$.

\begin{defn}
    Let $(\A,\T^n)$ be an arrangement for which $|\A| \geq n$. We say that $\A$ is \emph{spanning} if there exists a subset $I \subseteq \A$ such that $|I|=n$ we have $\bigcap_{i \in I} H_i \neq \emptyset$.
\end{defn}

Equivalently, after fixing a basis, we call an arrangement $\A$ spanning if there exists a subset of hyperplanes whose normal vectors span $\R^n$.

\begin{exmp}
    Take the arrangement $(\I,\T^n)$ where $\I = \{H_1,H_2,\ldots,H_n\}$ and each hyperplane $H_i = \ker(\theta_i)$ is defined by the map $$\theta_i: \T^n \to S^1; \; \theta(\Vec{x}) = \Vec{e}_i \cdot \Vec{x}.$$ We call such an arrangement a \emph{standard arrangement} $\I$ since each hyperplane is given by a standard basis vector. The standard arrangement is clearly spanning and in general position, since all $n$ hyperplanes intersect in a single point.
\end{exmp}

\begin{defn}
    A \emph{partially ordered set} (or a poset) is a set $P$ and a relation $\leq$ satisfying the following axioms: for all $x,y,z \in P$
    \begin{enumerate}
        \item (reflexivity) $x \leq x$,
        \item (anti-symmetry) if $x \leq y$ and $y \leq x$, then $x = y$,
        \item (transitivity) if $x \leq y$ and $y \leq z$, then $x \leq z$.
    \end{enumerate}
\end{defn}

\begin{defn}[\cite{stanley_2015}]
    Let $\A$ be an arrangement in a $k$-vector space $V$, and let $L(\A)$ be the set of connected components of non-empty intersections of hyperplanes in $\A$, including $V$ itself as the intersection over the empty set. Define $x \leq y$ in $L(\A)$ if $x \supseteq y$. We call $L(\A)$ the \emph{intersection poset} of $\A$.
\end{defn}

\begin{exmp}
Going back to $\T^n$, the definitions above can be adapted to work in $\T^n$ in a similar way they work in $\R^n$. Let $(\A,\T^2)$ be an arrangement of two hyperplanes $H_1$ and $H_2$ given by vectors $\langle 1,2 \rangle$, and, $\langle 1,-2 \rangle$ respectively, as shown in \cref{fig:firstexampleoflinearrangement}. The arrangement $\A$ is in general position, since $\dim(H_1 \cap H_2) = 0 = \dim(\T^2) - |\A|$. \cref{fig:exmpofanarrangementandaposet} shows the intersection poset $L(\A)$ of the arrangement $\A$. Note that the maximum element $\emptyset$, the intersection over the empty set, is usually omitted from $L(\A)$.

\begin{figure}[h!]
    \centering
    \[\begin{tikzcd}
    	{V_1} && {V_2} && {V_3} && {V_4} \\
    	& {H_1} &&&& {H_2} \\
    	&&& {\T^n}
    	\arrow[from=3-4, to=2-2]
    	\arrow[from=3-4, to=2-6]
    	\arrow[from=2-2, to=1-1]
    	\arrow[from=2-6, to=1-1]
    	\arrow[from=2-2, to=1-3]
    	\arrow[from=2-6, to=1-3]
    	\arrow[from=2-2, to=1-5]
    	\arrow[from=2-6, to=1-5]
    	\arrow[from=2-2, to=1-7]
    	\arrow[from=2-6, to=1-7]
    \end{tikzcd}\]
    \caption{An arrangement $(\A,\T^2)$ and its intersection poset $L(\A)$.}
    \label{fig:exmpofanarrangementandaposet}
\end{figure}
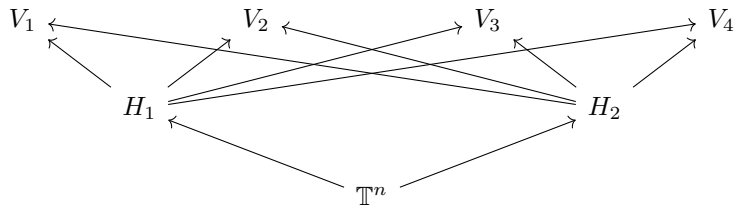
\end{exmp}

%%%%%%%%%%%%%%%%%%%% new section %%%%%%%%%%%%%%%%%%%%%%%%%%%%
\section{Deletion-restriction relations}
\label{sec:sectionfour}

In this section, we discuss the toric analogue of the deletion-restriction relations (\cite{stanley_2015}) and its application to proving the symmetry between the number of $k$-flats and the number of $(n-k)$-flats of an arrangement $(\A,\T^n)$.

Let $I$ be the indexing set for hyperplanes in an arrangement $\A$.
\begin{defn}
\label{defn:regionsofanarrangement}
The complement $\mathbb{T}^n \setminus \bigcup _{i \in I} H_i$ consists of a disjoint union of finitely many connected open regions called \emph{regions of an arrangement} \(\mathcal{A}\), where is not necessarily spanning or in general position. We will denote the set of such connected components $R(\A)$. Set $r(\A) := |R(\A)|$.
\end{defn} 

From the previous few examples, we have seen that regions of an arrangement $\A$ (not necessarily spanning or in general position) can be one of the two following types.

\begin{enumerate}
    \item The region $R_0 \in R(\A)$ is a convex polytope\footnote{Recall that a \emph{convex polytope} is the convex hull of a finite point set $P \subseteq \T^n$.} which implies the fundamental group of $R_0$ is $\pi_1(R_0) \cong \{e\}$.
    \item For a region $R_0 \in R(\A)$, its fundamental group is infinite cyclic, that is, $\pi_1(R_0) \cong \Z^k$ for some $k \in \N$.
\end{enumerate}

This is unlike real hyperplane arrangements, where every region is simply connected. However, the following lemma asserts that when we assume arrangements to be in general position and spanning, it is not possible to have regions with an infinite fundamental group.

\begin{lem}
\label{lem:generalpositionmeansflatsareconvex}
    Suppose $(\A,\T^n)$ is a spanning arrangement in general position with $|\A| \geq n$. Then every element of $R(\A)$ is simply connected.
\end{lem}

\begin{proof}
    Suppose a spanning subset of hyperplanes $H_1, H_2, \ldots, H_n$ is given by maps $\theta_1,\theta_2,\ldots,\theta_n$, respectively. Since we assume $|\A| \geq n$, the set of hyperplanes must span the torus $\T^n$. In other words, \[\Theta: \T^n \xrightarrow{(\theta_1,\;\theta_2,\;\ldots\;,\;\theta_n)} (S^1)^n\] is a covering space because, by assumption, the spanning set of hyperplanes in $\T^n$ is locally homeomorphic to $(S^1)^n$. Then $\Theta$ induces maps $$\Theta_{\ast}: \T^n \setminus \bigcup_{i \leq n} H_i \xrightarrow{(\theta_1,\;\theta_2,\;\ldots\;,\;\theta_n)} (S^1 \setminus \bigcup_{i \leq n} b_i)^n,$$ which is also a covering space. Furthermore, $(S^1 \setminus \bigcup_{i \leq n} b_i)^n \cong I^n$ where $I$ is the unit interval, so the connected components of $\T^n \setminus \bigcup_{i \in I} H_i$ are convex polytopes, since $\pi_1(I^n) \cong \{e\}$. The proof is complete after noting that removing more than $n$ hyperplanes leaves the regions convex.
\end{proof}

Therefore, we can get away with considering simply connected polytopes which arise in $\A$.

\begin{defn}
    Let $(\A, \T^n)$ be an arrangement. If $\mathcal{B} \subseteq \A$ is a subset, then $(\mathcal{B}, \T^n)$ is called a \emph{ subarrangement}. For $X \in L(\A)$ define a subarrangement $\A_X$ of $\A$ by \[\A_X = \{H \in \A \colon X \subseteq H \}. \] Define an arrangement $(\A^X,X)$ in $X$ by \[\A^X = \{ X \cap H \colon H \in \A \setminus \A_X \; \text{and} \; X \cap H \neq \emptyset \}.\] We call $\A^X$ the \emph{restriction} of $\A$ to $X$.
\end{defn}

Choose a hyperplane $H_0 \in \mathcal{A}$. Let $\mathcal{A}' = \mathcal{A} \setminus \{ H_0 \}$ be the \emph{deletion of $H_0$} and we call $\mathcal{A}'' = \A^{H_0}$ to be the \emph{restriction to $H_0$}. The deletion $\A'$ and the restriction $\A''$ are also toric arrangements. We call $(\mathcal{A}, \mathcal{A}', \mathcal{A}'')$ a triple of arrangements with distinguished hyperplane $H_0$. See an example of such a triple on $\T^2$ in \cref{fig:deletionrerstrictionexample}.

\begin{exmp}
    Suppose we have a line arrangement $(\A,\T^2)$, with $|\A|\geq n$. We want to describe the sets of vertices, edges, and faces using intersections of hyperplanes. Firstly, we have from \cref{defn:regionsofanarrangement} the $2$-dimensional objects which we will call $2$-flats. So we define $\Sigma_2(\A)$ to be the set $$ \pi_0 \left( \T^2 \setminus \bigcup_{i \in I} H_i\right).$$ Similarly, we can also describe the $0$-dimensional objects which we will call $0$-flats, $$\Sigma_0(\A) = \bigcup_{i,j \in I} \pi_0(H_i \cap H_j),$$ because a $0$-flats arise as intersections of 2 hyperplanes in $\A$. Finally, the $1$-dimensional objects of $\A$ must be subsets of the hyperplanes themselves. Moreover, we want them to be the edges which we get from the intersections of hyperplanes. Therefore, we write $$\Sigma_1(\A) = \bigcup_{i \in I} \pi_0(H_i \setminus \Sigma_0(\A)).$$
\end{exmp}

\begin{defn}
\label{defn:kflats}
    Let $(\A,\T^n)$ be an arrangement and $I$ an indexing set for hyperplanes in $\A$. We define the set of \emph{$0$-flats} of $\A$ to be 
    \[\Sigma_0(\A) = \bigcup_{\substack{S \subseteq I \\ |S|=n}} \pi_0 \left( \bigcap_{i \in S} H_i \right) . \]
    A $0$-flat is an element of $\Sigma_0$. For $0 < k < n$, we inductively define a \emph{$k$-flat} $\sigma_k$ of an arrangement $\A$ to be an element of the set
    \[\Sigma_k(\A) = \bigcup_{\substack{S \subseteq I \\ |S|=n-k}} \pi_0 \left( \bigcap_{i \in S} H_i \setminus \bigcup_{j < k}\Sigma_j\right).\]
    Finally, we define \[\Sigma_n(\A) = \pi_0\left( \T^n \setminus \bigcup_{i \in I} H_i \right). \]
    Let us denote the number of $k$-flats $|\Sigma_k|$ of an arrangement $\A$ by $f_k(\A)$.
\end{defn}

We take $R(\A)$ to mean the same objects as the $n$-flats $\Sigma_n$ when $\A$ is spanning and in general position by \cref{lem:generalpositionmeansflatsareconvex}. We will also use the convention that there are no $(-1)$-flats when considering any arrangement $(\A,\T^n)$, in other words $f_{-1}(\A) = 0$.\footnote{Note that the convention, such as in \cite{stanley_1991}, is to consider the empty set as a $(-1)$-flat of $\A$, so $f_{-1}(\A) = 1$.}

\begin{defn}
    Let $(\A,\T^n)$ be an arrangement. The \emph{$f$-vector $\Vec{f}(\A)$} of an arrangement $\A$ is defined as \[\Vec{f}(\A) = 
        (f_0(\A), f_1(\A), \ldots, f_n(\A)).\]
\end{defn}

We have previously defined an intersection poset $L(\A)$ for an arrangement $(\A,\T^n)$. Since flats arise as intersections of hyperplanes, we define the flat poset $\mathcal{L}(\A)$.

\begin{defn}[\cite{orlik1992arrangements}]
    The \emph{flat poset} $\mathcal{L}(\A)$ of the arrangement $\A$ is the poset of flats of $\A$ ordered by reverse boundary inclusion, that is, $\tau \ledot \sigma$ if $\overline{\tau} \subseteq \sigma$.
\end{defn}

The flat poset is a sharper invariant than the intersection poset. Define a map $\zeta: \mathcal{L}(\A) \to L(\A)$ by $P \mapsto |P|$, where $|P|$ is the support (i.e., the affine linear span) of $P$. The map $\zeta$ is order-preserving and surjective. It is true that if two arrangements $\A_1, \A_2$ have the same flat poset, then $L(\A_1) = L(\A_2)$. The converse does not hold in general.

\begin{exmp}
    We continue the example of a line arrangement $\A$ with $H_1 = \langle 1,2 \rangle$ and $H_2 = \langle 1,-2\rangle$. The flats of the arrangement are vertices, edges and faces which we will label $V,E,$ and $F$ respectively. Ordering the flats by reverse boundary inclusion, we get the flat poset in \cref{fig:flatposet}. In this example, we leave out the maximal element $\emptyset$. 

    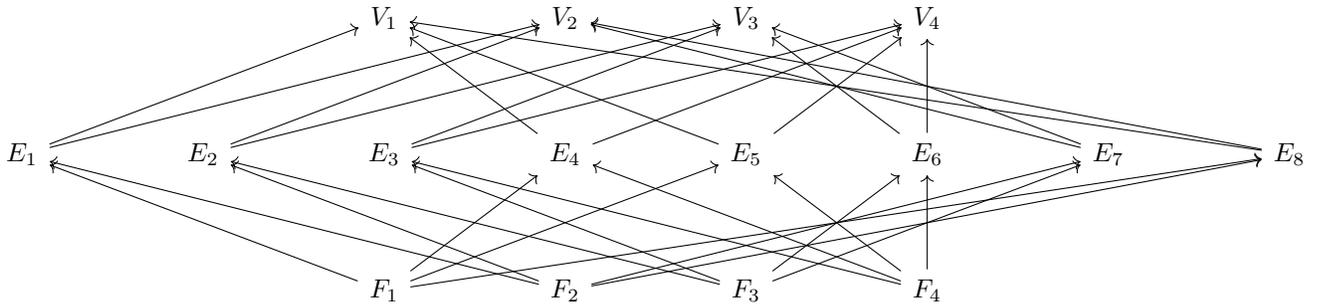
\begin{figure}[ht]
        \centering
        \[\begin{tikzcd}
	&&&& {V_1} && {V_2} && {V_3} && {V_4} \\
	\\
	{E_1} && {E_2} && {E_3} && {E_4} && {E_5} && {E_6} && {E_7} && {E_8} \\
	\\
	&&&& {F_1} && {F_2} && {F_3} && {F_4}
	% &&&&&&& {\T^2}
	% \arrow[from=7-8, to=5-5]
	% \arrow[from=7-8, to=5-7]
	% \arrow[from=7-8, to=5-9]
	% \arrow[from=7-8, to=5-11]
	\arrow[from=5-5, to=3-1]
	\arrow[from=5-5, to=3-7]
	\arrow[from=5-5, to=3-15]
	\arrow[from=5-5, to=3-9]
	\arrow[from=5-7, to=3-1]
	\arrow[from=5-7, to=3-13]
	\arrow[from=5-7, to=3-15]
	\arrow[from=5-7, to=3-3]
	\arrow[from=5-9, to=3-3]
	\arrow[from=5-9, to=3-5]
	\arrow[from=5-9, to=3-11]
	\arrow[from=5-9, to=3-13]
	\arrow[from=5-11, to=3-7]
	\arrow[from=5-11, to=3-9]
	\arrow[from=5-11, to=3-5]
	\arrow[from=5-11, to=3-11]
	\arrow[from=3-1, to=1-5]
	\arrow[from=3-1, to=1-7]
	\arrow[from=3-3, to=1-7]
	\arrow[from=3-3, to=1-9]
	\arrow[from=3-5, to=1-9]
	\arrow[from=3-5, to=1-11]
	\arrow[from=3-7, to=1-11]
	\arrow[from=3-7, to=1-5]
	\arrow[from=3-9, to=1-5]
	\arrow[from=3-9, to=1-11]
	\arrow[from=3-11, to=1-11]
	\arrow[from=3-11, to=1-9]
	\arrow[from=3-13, to=1-9]
	\arrow[from=3-13, to=1-7]
	\arrow[from=3-15, to=1-7]
	\arrow[from=3-15, to=1-5]
    \end{tikzcd}\]
        \caption{An example of a flat poset $\mathcal{L}(\A)$.}
        \label{fig:flatposet}
    \end{figure}
\end{exmp}

Using the definitions above, it is possible to enumerate the flats, thus determining the $f$-vector, of special arrangements. Before doing so, we prove the following lemma.

\begin{lem}
\label{lem:changeofbasisontorus}
Consider the quotient $\pi: \R^n \to \T^n$ and let $\hat{\varphi}$ be a map $\hat{\varphi}: \mathbb{R}^n \to \mathbb{R}^n$ that sends $\boldsymbol{\theta} \in \mathbb{Z}^n$ to $\hat{\varphi} (\boldsymbol{\theta}) \in \mathbb{Z}^n$. Then there exists a map $\varphi: \T^n \to \T^n$ such that the following diagram commutes.
\[\begin{tikzcd}
	{\R^n} && {\R^n} \\
	\\
	{\T^n} && {\T^n}
	\arrow["{\tilde{\theta}}", from=1-1, to=1-3]
	\arrow["\theta"', from=3-1, to=3-3]
	\arrow["\pi"', from=1-1, to=3-1]
	\arrow["\pi", from=1-3, to=3-3]
\end{tikzcd}\]
\end{lem}

\begin{proof}
It suffices to check that the maps are well-defined, that is, the action of $\mathrm{GL_n(\Z)}$ is well-defined. Let $A \in \mathrm{GL}_n(\Z)$ be an $n\times n$ invertible matrix with integer entries, and define $\varphi: \boldsymbol{\theta} \mapsto A\boldsymbol{\theta}$. Let $\left[ \boldsymbol{\theta}\right] \subset \T^n$ be an equivalence class of matrices. Consider two distinct representatives $\boldsymbol{\theta _1}, \boldsymbol{\theta _2} \in \left[ \boldsymbol{\theta}\right]$, so $\boldsymbol{\theta _1} - \boldsymbol{\theta _2} \in \Z ^n$. Then \(\left[ \varphi \left(\boldsymbol{\theta _1}\right)\right] = \left[ \varphi \left(\boldsymbol{\theta _2}\right)\right]\) is equivalent to \(A\boldsymbol{\theta _1} - A\boldsymbol{\theta _2} = A(\boldsymbol{\theta _1} - \boldsymbol{\theta _2})\). Hence, we have \(\boldsymbol{\theta _1} - \boldsymbol{\theta _2} \in \Z ^n\) by definition and multiplying by $A$ will result with a vector with integer entries.
\end{proof}

For the proof of the following lemma, we make use of the Smith normal form. The Smith normal form takes an arbitrary matrix $A \in \textrm{Mat}_{n \times m}(\Z)$ and factorises it as $A = NDM$ where $N \in \mathrm{GL}_n(\Z)$, $M \in \mathrm{GL}_m(\Z)$, and $D$ is a diagonal rectangular matrix. We state without a proof that any square matrix with entries in a principal integral domain admits a Smith normal form. \cref{lem:changeofbasisontorus} says, in essence, that it is legal to apply Smith normal form.

\begin{lem}
\label{lem:numberofkflatsinnicearragements}
    Let $(\A,\T^n)$ be an arrangement such that $\A=\{H_1,H_2,\ldots,H_n\}$ where each hyperplane is defined by its normal vector. Then we have \[f_k(\A)= |\det(\Theta)|\binom{n}{k},\] where $\Theta$ is the system of hyperplanes of $\A$.
\end{lem}

\begin{proof}
    First, we check the lemma for the standard arrangement $(\mathcal{I},\T^n)$, so we define each hyperplane $H_i$ by the kernel of the map $$\theta_i:\T^n \to S^1; \; \Vec{x} \mapsto \Vec{e}_i\cdot\Vec{x}.$$ Every $k$-flat in $\mathcal{I}$ is uniquely determined by choosing $n - k$ hyperplanes because any subset of hyperplanes in $\A$ intersect uniquely. Hence, $$f_k(\A)= \binom{n}{n-k} = \binom{n}{k}.$$ 
    In case the hyperplanes are not given by coordinate vectors, fix a basis and consider the system of hyperplanes $\Theta$ of $\A$. Using Smith normal form, we can write $\Theta = N\Theta'M$. Then the rows of $\Theta'$ are the normal vectors of multiples of coordinate hyperplanes. This also describes a covering map $q: \T^n \to \T^n$ given by $(\Vec{e_1}, \Vec{e_2}, \ldots, \Vec{e_n}) \mapsto (m_1\Vec{e_1},m_2\Vec{e_2},\ldots,m_n\Vec{e_n})$, where $m_i$ is the $i$-th diagonal entry of $\Theta'$. The fibre $q^{-1}(0)$ has cardinality $|\det(\Theta')|$, therefore $q$ is a covering of degree $|\det(\Theta')|$, so $$f_k(\A)= |\det(\Theta')|\binom{n}{k} = |\det(\Theta)|\binom{n}{k}.$$
\end{proof}

\begin{cor}
\label{lem:determinatgivesregions}
Let $(\A,\T^n)$ be an arrangement of $n$ hyperplanes $H_1,H_2,\ldots , H_n$ defined by their normal vector. Then the number of regions is given by the determinant of the system of hyperplanes $\det(\Theta)$.
\end{cor}

% \begin{figure}[h!]
%     \centering
%     \begin{tikzpicture}[scale=.7]
%         \draw  (-2.5,1) ellipse (2.5 and 0.5);
%         \node [circle,fill=red,outer sep=0,inner sep=1,minimum size=1] at (-5,1) {};
%         \node [circle,fill,outer sep=0,inner sep=1,minimum size=1] at (-3.2,0.5) {};
%         \node [circle,fill,outer sep=0,inner sep=1,minimum size=1] at (-2,0.5) {};
%         \node [circle,fill,outer sep=0,inner sep=1,minimum size=1] at (-1.1,1.4) {};
%         \node [circle,fill,outer sep=0,inner sep=1,minimum size=1] at (-3.9,1.4) {};
%         \node [outer sep=0,inner sep=0,minimum size=0] at (-2.6,-0.1) {$\mathcal{A}$};
%     \end{tikzpicture} \begin{tikzpicture}[scale=.7]
%         \draw  (-2.5,1) ellipse (2.5 and 0.5);
%         \node [circle,fill,outer sep=0,inner sep=1,minimum size=1] at (-3.2,0.5) {};
%         \node [circle,fill,outer sep=0,inner sep=1,minimum size=1] at (-2,0.5) {};
%         \node [circle,fill,outer sep=0,inner sep=1,minimum size=1] at (-1.1,1.4) {};
%         \node [circle,fill,outer sep=0,inner sep=1,minimum size=1] at (-3.9,1.4) {};
%         \node [outer sep=0,inner sep=0,minimum size=0] at (-2.6,-0.1) {$\A^\prime$};
%     \end{tikzpicture} 
    
%     \begin{tikzpicture}
%         \node [circle,fill=red,outer sep=0,inner sep=1,minimum size=1] at (-2,0.5) {};
%         \node [outer sep=0,inner sep=0,minimum size=0] at (-2,-0.1) {$\A^{\prime\prime}$};
%     \end{tikzpicture}
%     \caption{The triple $(\A,\A^\prime,\A^{\prime \prime})$ on $\T^1$.}
%     \label{fig:triplewithadistinguishedhyperplane}
% \end{figure}

\subsection{Generalised deletion-restriction and the symmetry of \texorpdfstring{$f$}{Generalised deletion-restriction and symmetry of f-vectors}-vectors}
Recall the following relation in case of real hyperplane arrangements between the regions of $\A, \A',$ and $\A''$.

\begin{thm}[\cite{stanley_2015}]
\label{thm:stanleydeletionrestriction}
    Let $(\A,\A^\prime, \A^{\prime\prime})$ be a triple of arrangements with distinguished hyperplane $H_0$. Then 
    \[f_n (\A) = f_n(\A') + f_{n-1} (\A^{\prime\prime}).\]
\end{thm}

The proof of \cref{thm:stanleydeletionrestriction} provided in \cite{stanley_2015} is adaptable to arrangements in general position on $\T^n$ as well. However, instead of only studying the deletion-restriction relation for regions of $\A$, we would like to generalise the theorem for other dimensional flats as well. 

\begin{exmp}
    Suppose that we have an arrangement $(\A,\T^2)$ with hyperplanes $H_0 = \langle 1,0 \rangle, H_1 = \langle -1,2 \rangle, H_2 = \langle 1,2 \rangle$. We let $H_0$ be a distinguished hyperplane and consider $\A', \A''$, see \cref{fig:deletionrerstrictionexample}. The objective is to count the number of $1$-dimensional flats recursively using $\A'$ and $\A''$. We can rewrite the set of $1$-flats $\Sigma_1$ of $\A$ as
    \[ \Sigma_1(\A) = \{ \sigma \in \Sigma_1(\A) \colon \sigma \cap H_0 = \emptyset \} \sqcup \{\sigma \in \Sigma_1(\A) \colon \sigma \cap H_0 \neq \emptyset, \sigma \nsubseteq H_0\} \sqcup \{ \sigma \in \Sigma_1(\A) \colon \sigma \subseteq H_0 \}, \]
    that is, we distinguish the $1$-flats of $\A$ which do not intersect $H_0$, the 1-flats which intersects $H_0$ in codimension 1 and the 1-flats which are properly contained in $H_0$.

    \begin{figure}[htt]
        \centering
                \begin{tikzpicture}
    	\definecolor{cadmiumgreen}{rgb}{0.0, 0.42, 0.24}
      \tikzstyle{invi} = [outer sep=0,inner sep=0,minimum size=0]
      \tikzstyle{pt} = [fill,circle,outer sep=0,inner sep=1.5,minimum size=1.5]
        \draw  (-4.5,3.5) node[invi] (v3) {} rectangle (-0.5,-0.5) node[invi] (v1) {};
        \draw[blue] (-4.5,-0.5) -- (-0.5,1.5);
        \draw[blue] (-4.5,1.5) node[invi] (v2) {} -- (-0.5,3.5);
        \draw[orange] (v1) -- (v2);
        \draw[orange] (-0.5,1.5) -- (v3);
        \draw[cadmiumgreen] (-3.5,3.5) -- (-3.5,-0.5);
        \draw  (1,3.5) node[invi] (v5) {} rectangle (5,-0.5) node[invi] (v4) {};
        \draw[orange] (v4) -- (1,1.5) node[invi] (v7) {};
        \draw[orange] (5,1.5) node[invi] (v6) {} -- (v5);
        \draw[blue] (1,-0.5) -- (v6);
        \draw[blue] (v7) -- (5,3.5);
        
        \node[pt] at (-3.5,0) {};
        \node[pt] at (-3.5,1) {};
        \node[pt] at (-3.5,2) {};
        \node[pt] at (-3.5,3) {};
        \node[pt] at (7.1,2.1) {};
        \node[pt] at (9,2.2) {};
        \node[pt] at (7.2,1.1) {};
        \node[pt] at (10.2,1.2) {};
        \node at (-2.6,-1) {$\mathcal{A}$};
        \node at (3.2,-1) {$\mathcal{A}'$};
        \node at (8.6,0.5) {$\mathcal{A}''$};
        \draw[cadmiumgreen] (8.5,1.6) ellipse (2.4 and 0.6);
        \node at (-3.3,1.2) {$b$};
	\node[pt] at (-2.5,0.5) {};
	\node at (-2.4,0.8) {$a$};
	\node[pt] at (3,0.5) {};
	\node at (3.1,0.8) {$a$};
	\node at (6.9,0.9) {$b$};
    \end{tikzpicture}
        \caption{The triple of arrangements $(\A,\A',\A'')$ with $H_0 = \langle 1,0 \rangle$ (green), $H_1 = \langle -1,2 \rangle$ (blue), $H_2 = \langle 1,2 \rangle$ (orange).}
        \label{fig:deletionrerstrictionexample}
    \end{figure}
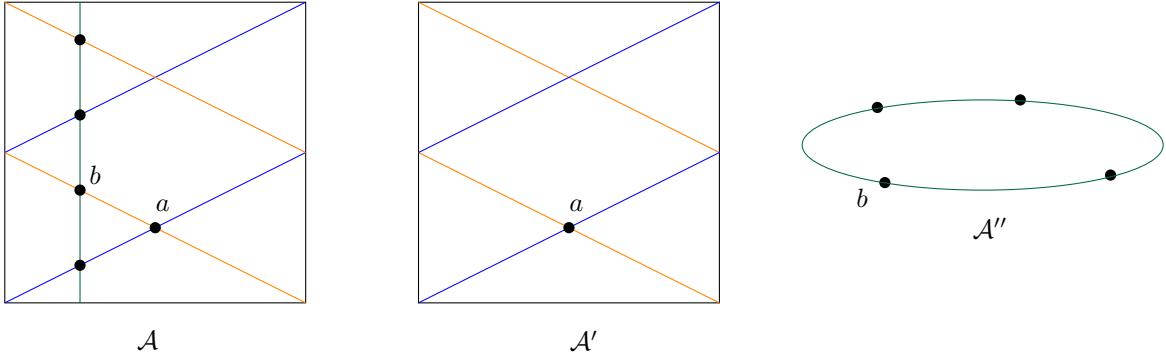

    Clearly, $1$-flats of $\A''$ is a subset of $1$-flats of $\A$ as these flats are not contained in $\A'$. We can see that the other $1$-flats of $\A$ arise as $1$-flats of $\A'$ which are either preserved or cut up by $H_0$. We can thus write the set of $1$-flats of $\A'$ as
    \[\Sigma_1 (\A') = \{\sigma \in \Sigma_1 (\A') \colon \sigma \cap H_0 = \emptyset\} \sqcup \{\sigma \in \Sigma_1(\A') \colon \sigma \cap H_0 \neq 0 \}.\]
    The set of $1$-flats of $\A'$ which have an empty intersection with $H_0$ are also $1$-flats of $\A$. On the other hand, the $1$-flats with a non-empty intersection with $H_0$ are the flats cut up by $H_0$. Since a $1$-flat is homeomorphic to an (open) interval $I$ in $\R$, we know that $m$ points, corresponding to $m$ intersections with $H_0$, divide $I$ into $m+1$ intervals. The intersection information is encoded as $0$-flats of $\A''$ by definition. Also, because we now have $m+1$ intervals, we can identify $m$ intervals with a $m$ subset of $0$-flats of $\A''$ and the remaining interval is identified with the original $1$-flat of $\A'$. This holds for every flat of $\A'$ which has a non-empty intersection, so  finally we identify
    \[ \Sigma_1(\A) \cong \Sigma_1(\A') \sqcup \Sigma_1(\A'') \sqcup \Sigma_0(\A''), \] which implies \[f_1(\A) = f_1(\A') + f_1(\A'') + f_0(\A'').\]
\end{exmp}

\begin{lem}[Generalised deletion-restriction]
\label{lem:deletionrestriction}
Let $(\mathcal{A}, \mathcal{A}^\prime, \mathcal{A}^{\prime\prime})$ a triple of spanning arrangements in general position, such that every flat is simply connected, with a distinguished hyperplane $H_0$ and assume that $|\A| \geq n + 1$. Then
\[f_k (\A) = f_k (\A') + f_k (\A'') + f_{k-1}(\A'').\]
\end{lem}

\begin{proof}
    Let the distinguished hyperplane $H_0$ be defined by the group homomorphism $\theta: \T^n \to S^1$. Then we can decompose the $k$-flats of $\A$, $\Sigma_k(\A)$, as a disjoint union,
    \begin{equation*}
         \Sigma_k(\A) = \overbrace{\{\sigma \in \Sigma_k(\A) \colon \sigma \cap H_0 = \emptyset \}}^{(1)} \sqcup \overbrace{\{\sigma \in \Sigma_k(\A) \colon \sigma \cap H_0 \neq \emptyset, \sigma \nsubseteq H_0 \}}^{(2)} \sqcup \overbrace{\{\sigma \in \Sigma_k(\A) \colon \sigma \subseteq H_0\}}^{(3)},
    \end{equation*}
    where we keep track of the sets by numbering throughout the proof.
    
    Since $\A$ is in general position, it follows that any $k$-flat $\sigma \in \Sigma_k(\A)$ belongs to exactly one of the sets listed above.

    We immediately have that a $k$-flat properly contained in $H_0$ must be a $k$-flat of $\A''$ by definition, so \[\overbrace{\{\sigma \in \Sigma_k(\A) \colon \sigma \subseteq H_0\}}^{(3)} = \Sigma_k(\A'').\]

    The $k$-flats of $\A$ which are not contained in $\A''$ must be contained in $\A'$. So we examine $\Sigma_k(\A')$. By adding $H_0$ to $\A'$, a $k$-flat of $\A'$ is either preserved or gets cut up by $H_0$ into more $k$-flats. Thus, the set $\Sigma_k(\A')$ can be written as a disjoint union \[ \Sigma_k(\A') = \overbrace{\{\sigma \in \Sigma_k(\A') \colon \sigma \cap H_0 = \emptyset \}}^{(4)} \sqcup \overbrace{\{\sigma \in \Sigma_k(\A') \colon \sigma \cap H_0 \neq \emptyset\}}^{(5)}. \]
    The set of $k$-flats of $\A'$ with an empty intersection with $H_0$ is in bijection with the set of $k$-flats of $\A$ with an empty intersection with $H_0$, therefore \[ \overbrace{\{ \sigma \in \Sigma_k(\A') \colon \sigma \cap H_0 = \emptyset\}}^{(4)} = \overbrace{\{\sigma \in \Sigma_k(\A) \colon \sigma \cap H_0 = \emptyset \}}^{(1)}.\]

    Now it is left to investigate the set $(2)$, the $k$-flats of $\A$ which have a non-empty intersection with $H_0$ but are not properly contained in $H_0$. We claim that such flats can be labelled by either a $k$-flat of $\A'$ or a $(k-1)$-flat of $\A''$. To see this, let $\pi:\T^n \to \R^n$ be a covering space and $\tilde{\theta}:\R^n \to \R$ be the lift of $\theta$. For a $k$-flat $\sigma$, we will let $\tau$ be its \emph{boundary} if $\sigma \ledot \tau$ in $\mathcal{L}(\A)$ and $\dim(\tau) = \dim(\sigma) - 1$. Define a boundary $\tau$ of a lift of a flat $\pi^{-1}(\sigma)$ in $H_0$ to be positive if for all paths $\gamma:[0,1] \to \R^n$ contained in $\pi^{-1}(\sigma)$ ending at $\tau$ with $\gamma' = 0$ we have $(\tilde{\theta}(\gamma))' > 0$.

    Fix a flat $\sigma \in \Sigma_k(\A)$ and suppose, for sake of contradiction, that its lift $\pi^{-1}(\sigma)$ has two distinct positive boundaries, say $\tau_1$ and $\tau_2$. Since $\sigma$ is convex by \cref{lem:generalpositionmeansflatsareconvex}, $\pi^{-1}(\sigma)$ is convex and there exists a positive linear path $\gamma$ such that $\gamma(0) \in \tau_1$ and $\gamma(1) \in \tau_2$ which is contained in $\pi^{-1}(\sigma)$. But then the opposite path $\gamma'$ such that $\gamma'(0) \in \tau_2$ and $\gamma'(1) \in \tau_1$ cannot be positive, so $\tau_1$ is not positive which is a contradiction. Thus, there is at most one positive boundary of $\sigma$ in $H_0$.

    Each $k$-flat $\sigma$ in $\{\sigma \in \Sigma_k(\A') \colon \sigma \cap H_0 \neq \emptyset\}$ can be written as a disjoint union of a finite chain of $k$-flats $\sigma^1, \sigma^2, \ldots, \sigma^m \in \Sigma_k(\A)$ so that $\sigma^i \subseteq \sigma$ for every $i \in \{1,2\ldots,m\}$. There exists a $\sigma^0$ such that $$\tilde{\theta}(\pi^{-1}(\sigma^i)) \leq \tilde{\theta}(\pi^{-1}(\sigma^0)),$$ for all $i \in I$. We claim that $\sigma^0$ is the unique flat contained in $\sigma$ which is not positive. As done previously, suppose there are two distinct $k$-flats $\sigma^0$ and $\sigma^1$ with non-positive boundaries in $H_0$ for contradiction. Then $\sigma$ must have two distinct positive boundaries in $H_0$ which is a contradiction by the above.

    Thus, for any $k$-flat $\sigma$ in the set $\{\sigma \in \Sigma_k(\A') \colon \sigma \cap H_0 \neq \emptyset\}$ can be written as \[\sigma = \bigsqcup_{i \in I} \sigma^i,\] where $\sigma^i \in \Sigma_k(\A)$ and $\sigma^i$ is not positive for exactly one index $i \in I$. The flats $\sigma^i$ with a positive boundary are in correspondence with the $(k-1)$-flats of $\A''$ which are contained in $\sigma$, while the $\sigma^i$ without a positive boundary are in correspondence with $\sigma$. We conclude that 
    \[(2) = (4) \sqcup \Sigma_{k-1}(\A'').\] 
    Therefore the $k$-flats of $\A$ can be written as
    \[\Sigma_k(\A) \cong \Sigma_k(\A') \sqcup \Sigma_k(\A'') \sqcup \Sigma_{k-1}(\A''), \] and enumerating the sets yields \[ f_k(\A) = f_k(\A') + f_k(\A'') + f_{k-1}(\A''). \]
\end{proof}

We conjecture the generalised deletion-restriction applies to arrangements containing non-simply connected regions as well, which we demonstrate in the following simple example.

\begin{exmp}
     Suppose we have the standard arrangement $\I$ on $\T^2$ and we want to count the number of edges using the generalised deletion-restriction. Then $f_1(\I') = 0$ because there are no simply-connected flats in $\I'$. However, $\I''$ is essentially an arrangement on $\T^1$ of one hyperplane hence $f_1(\I'') = f_{0}(\I'') = 1$, so the relation rightly yields $f_1(\I) = 2$.
\end{exmp}

We make use of \cref{lem:deletionrestriction} to prove the main result of the paper.

\begin{thm}[Symmetry of $f$-vectors]
\label{prop:facetsduality}
Suppose we have a hyperplane arrangement $(\A,\T^n)$ such that $|\A|\geq n$. There is a bijection between the $k$-flats and $(n-k)$-flats of $\A$, $$\Sigma_k(\A) \cong \Sigma_{n-k}(\A).$$ In particular, \(f_k (\A) = f_{n-k} (\A)\).
\end{thm}

\begin{proof}
We proceed by induction on the dimension of the ambient space first. For the base case, consider any arrangement on $\T^0$ for which the statement holds trivially.

We do a second induction on the number of hyperplanes in $\A$. Using the same notation as the proof of \cref{lem:numberofkflatsinnicearragements}, we have the case of $n$ hyperplanes on $\T^n$, $$f_{n-k} = |\det(\Theta')|\binom{n}{k} = |\det(\Theta)|\binom{n}{k} = f_k(\A),$$ because $\Theta':\T^n \to \T^n$ is a $|\det(\Theta')|$-fold cover. We can thus assume that the statement holds for arrangements $(\A,\T^n)$ with $|\A| = n$. 

Now suppose the induction hypothesis holds for all arrangements $|\A| < m$ on $\T^{l}$ with $l \leq n$. Let $(\A,\T^n)$ be an arrangement with $|\A| = m + 1$. We investigate what happens when we remove a hyperplane $H_{m+1}$ from $\A$, so that $|\A'|=m$. By \cref{lem:deletionrestriction}, we write
\begin{equation*}
    \begin{split}
        f_k (\A) &= f_k(\A') + f_k(\A'') + f_{k-1}(\A''),\\
        f_{n-k} (\A) &= f_{n-k}(\A') + f_{n-k}(\A'') + f_{n-k-1}(\A'').
    \end{split}
\end{equation*}
Clearly, $f_k(\A') = f_{n-k}(\A')$ because $|\A'|=m$. Moreover, $\A''$ is an arrangement of (not necessarily simple) hyperplanes on $\T^{m-1}$ with $|\A''| \leq m-1$. So by induction, we also have \[ f_{k}(\A'') = f_{n-k - 1}(\A'') \; \text{and} \; f_{k-1}(\A'') = f_{n-k}(\A''), \] and the relation holds.
\end{proof}

As a sanity check, we make sure that we have not broken any topological invariants of $\T^n$, namely the Euler characteristic. Recall the Euler characteristic of an arrangement $(\A, V)$, where $\dim(V) = n$, is defined as $$\sum_{i = 0}^{n} (-1)^i f_i(\A).$$ In particular, for $V = \T^n$, we have $\sum_{i = 0}^{n} (-1)^i f_i(\A) = 0$. For odd $n$, it is easy to see that having $f_k(\A) =f_{n-k}(\A)$ respects the alternating sum equalling 0.

%%%%%%%%%%%%%%%%%%%% new section %%%%%%%%%%%%%%%%%%%%%%%%%%%%
\section{Regions of toric arrangements}
\label{sec:sectionfive}

\subsection{Counting the regions} Previously, we have defined regions to be the top dimensional objects of an arrangement $\A$. We have also seen that regions can only either be simply connected or have an infinite cyclic fundamental group. In the case regions of an arrangement $(\A,\T^n)$ of $n$ hyperplanes are simply connected, we have a definite answer in \cref{lem:numberofkflatsinnicearragements}. However, if we consider arrangement on $\T^n$ with strictly less than $n$ hyperplanes, we get a \emph{degenerate toric arrangement} due to the following lemma.

\begin{lem}
\label{lem:degeneratearrangement}
    Let $(\A,\T^n)$ be an arrangement with $|\A| = m$ such that $m < n$. Then $f_k(\A)= 0$ for all $k \in \{ 0, 1, \ldots, n \}$.
\end{lem}
\begin{proof}
    Since $m < n$, we cannot obtain any complete $n$-intersections in $\A$, i.e. there cannot exist any $0$-dimensional flats. This inductively implies that there are no $k$-flats, hence $f_k(\A)= 0$ for all $k \in \{ 0, 1, \ldots, n \}$.
\end{proof}

As a consequence of \cref{lem:degeneratearrangement}, from this point we will only consider arrangements on $\T^n$ with at least $n$ hyperplanes.

We have already established a result concerning enumerating regions for hyperplane arrangements of $n$ hyperplanes in \cref{lem:numberofkflatsinnicearragements}. To give a geometric intuition of the idea, we use the symmetry of $f$-vectors, namely $f_0(\A) = f_n(\A)$. In the toy example of $\T^2$, we used the fact that a normal vector $\langle a,b\rangle$ of a hyperplane on $\T^2$ can be written as $a\langle 1,0\rangle + b\langle 0,1\rangle$, counted the number of line intercepts and then applied the symmetry. Continuing the example on $\T^2$, let $H_1 = \langle a, b \rangle$ and $H_2 = \langle c, d \rangle$ be hyperplanes given by their normal vectors, as usual. By the definition of a hyperplane, we have
\begin{equation*}
\begin{split}
    H_1 = \ker \left( \theta_1: \T^2 \to S^1; \; \begin{pmatrix} x_1 \\ x_2 \end{pmatrix} \mapsto ax_1 + bx_2 \right),\\
    H_2 = \ker \left( \theta_2: \T^2 \to S^1; \; \begin{pmatrix} x_1 \\ x_2 \end{pmatrix} \mapsto cx_1 + dx_2 \right),
\end{split}
\end{equation*}
Clearly the number of intersections, or 0-flats, is given by \(|H_1 \cap H_2|\) which is the same as the cardinality of the kernel of the direct product of group homomorphisms $\theta_1 \times \theta _2$, that is \(\ker (\T ^2 \xrightarrow[]{\theta _{1},\theta_2} S^1 \times S^1)\). Furthermore, $|\ker(\theta_1 \times \theta _2)|$ is given by the number of lattice point in $\Z^2$ inside\footnote{Since we are working on $\T^2$, we count lattice points on the normal vectors and inside the parallelogram.} the parallelogram determined by the normal vectors of the hyperplanes. % should the formula be stated?
Pick's theorem then says that the number of lattice points is equal to the area, which is precisely the determinant. Therefore, by applying \cref{prop:facetsduality}, we have that the number of regions of $\A$ is given by \[|\det\begin{pmatrix}
    a&b\\c&d
\end{pmatrix}|.\]
We want to employ a similar idea for a general case on $\T^n$, in particular, we will use Smith normal form to change bases of hyperplanes. For every $i \leq n$, consider the set of hyperplanes that are defined as kernels of group homomorphisms

\begin{equation*}
        H_i = \ker \left( \theta_i: \T^n \to S^1; \; \begin{pmatrix} x_1 \\ \vdots \\ x_n \end{pmatrix} \mapsto a_{i1}x_1 + \ldots + a_{in}x_n \right).
\end{equation*}
The direct product $\theta _N = \theta _1 \times \ldots \times \theta _n$ is a group homomorphism with \(\ker (\theta _N) = \bigcap _{i = 1} ^n \ker (\theta _i)\). In other words, we are looking for all vectors $\boldsymbol{x} \in \R ^n$ such that

\begin{equation*}
    \begin{pmatrix} 
    a_{11} & \ldots & a_{1n} \\
    \vdots & \ddots & \vdots \\
    a_{n1} & \ldots & a_{nn}
    \end{pmatrix} 
    \begin{pmatrix} x_1 \\ \vdots \\ x_n \end{pmatrix} = \begin{pmatrix} k_1 \\ \vdots \\ k_n \end{pmatrix},
\end{equation*}
where each $k_i \in \Z$ as such vectors $\Vec{x}$ will be elements of the kernel. Hence, finding the number of elements in $\ker (\theta _N)$ is equivalent to searching for the number of lattice points inside the parallelepiped given rows of $ \Theta = \left( H_1 | \ldots | H_n \right) ^T$. Since any matrix over $\Z$ has a Smith normal form, we consider the Smith normal form of $\Theta$ which we will call $\Theta'$. By \cref{lem:changeofbasisontorus}, the Smith normal form of $\Theta$ now represents a rectangular parallelepiped. The problem is invariant under translation and rotation, so we may identify one vertex of the rectangular parallelepiped to the origin and the row vectors of $\Theta'$ with the axes. Note that we are only interested in counting the lattice points inside the half-open parallelepiped. We can thus associate to each unit of volume a point, which gives a bijection between the number of lattice points and the volume of the parallelepiped, i.e. $|\det(\Theta)|$. The result then follows as an application of \cref{prop:facetsduality}.

\begin{cor}
\label{cor:regionsofspanningarrangements}
    Let $(\A,\T^n)$ be an arrangement of $m$ hyperplanes such that $m >n$. To each hyperplane $H_i$ associate its normal vector $\Vec{a_i}$. Then the number of $n$-dimensional flats is given by the sum over all $n$-tuples of $\A$, \[ f_n(\A) = \sum_{\substack{S \subseteq \{1,2,\ldots,m\}\\ |S|=n}} |\det(\Vec{a_{i_1}} \; |\; \Vec{a_{i_2}}\; |\; \ldots \; |\; \Vec{a_{i_n}}\;)|. \]
\end{cor}

\begin{proof}
    The arrangement $\A$ being in general position implies that the 0-flats of $\A$ do not appear with multiplicities. If we take two subsets of $n$ hyperplanes such that they differ by at least one, their intersection points are disjoint. Therefore, we can enumerate the number of 0-flat of all hyperplane subsets of cardinality $n$ which yields the number of $n$-flats of $\A$.
\end{proof}

\begin{rem}
\label{prop:simplecaseparallel}
    The definition of regions $R(\A)$ does not exclude arrangements $\A$ which are not spanning. Let $\A$ be an arrangement on $\T^n$ of $m$ hyperplanes such that the hyperplanes are pairwise parallel. Clearly, $m$ points divide $\T^1 = S^1$ into $m$ segments. If we interpret $\T^n = S^1 \times S^1 \times \ldots \times S^1$, then the hyperplanes are parallel, and we may regard them as a subset of $$\underbrace{S^1 \times S^1 \times \ldots \times S^1}_{n-1},$$ and they only intersect one of the $S^1$ in the product. The single $n$-intersection of each hyperplane then defines one region, therefore $r(\A) = m$. However, since all hyperplanes are parallel, all the regions have a non-trivial fundamental group, so by our definition of flats, $f_i = 0$ for $i \in \{0,1,\ldots,n\}$.
\end{rem}

\subsection{Analogue of Zaslavky's theorem}

Another possible way to count the regions of an arrangement $(\A,\T^n)$ is to use the characteristic polynomial of $\A$. In the context of $\R^n$, this was proven due to T. Zaslavsky in 1975.

\begin{thm}[Zaslavsky]
\label{thm:zaslavsky}
    Let $(\A,\R^n)$ be a hyperplane arrangement. Then $$r(\A) = (-1)^n\chi_\A(-1).$$
\end{thm}

There exists a toric analogue of \cref{thm:zaslavsky} that has been proven; see Theorem 3.6 in \cite{Ehrenborg_2009}. However, we provide another type of proof using the fact that on general arrangements $(\A,\T^n)$ we apply $f_0 (\A) = f_n (\A)$ given by \cref{prop:facetsduality}. Of course, in our case, we still assume that $\A$ is spanning and in general position.

\begin{defn}
    Let $P$ be a locally finite poset. Define a function $\mu = \mu_P: \Int(P) \to \Z$, where $\Int(P)$ is the set of all closed intervals of $P$, called the \emph{Möbius function} of $P$ by the conditions:
    \begin{align*}
        &\mu(x,x) = 1, \quad \text{for all $x \in P$},\\
        &\mu(x,y) = - \sum_{x \leq z < y} \mu(x,z), \quad \text{for all $x<y$ in $P$.}
    \end{align*}
\end{defn}

If $P$ has a minimum element $\Hat{0}$, we write $\mu(x) = \mu(\Hat{0},x)$. We will also denote by $\Int(x)$ the set of closed intervals $[\Hat{0},x]$.

\begin{lem}
    For $x \in L(\A)$, the restricted poset $\Int(x) \hookrightarrow L(\A)$ is a Boolean lattice.
\end{lem}

\begin{proof}
    Let $L(\A)$ be an intersection poset of an arrangement $(\A,\T^n)$ where $\A = \{ H_1,H_2,\ldots, H_m\}$. Also let $S \subseteq \{1,2,\ldots, m\} = \mathbf{m}$. By Proposition 2.3. in \cite{stanley_2015} the poset $L(\A)$ is a meet-semilattice, i.e. for every pair of elements $x,y \in L(\A)$ there exists a meet $x \land y \in L(\A)$. Since $\mu$ depends only on $\Int(x)$, we will only want to consider $\Int(x)$ for any $x \in L(\A)$. But $x$ is a maximal element for $\Int(x)$, so $\Int(x)$ must be a lattice.
    
    Define $$x_S = \bigcap_{i\in S}H_i.$$ Then $L(\A)$ is a set of (non-empty) $x_S$ for all $S$ ordered by reverse set-inclusion. Let $2^{\mathbf{m}}$ denote the Boolean poset that arises from the subsets of the set $\mathbf{m}$ and consider an inclusion map $\Int(x) \hookrightarrow 2^{\mathbf{m}}$ given by $x_S \mapsto S$. Note that $x$ corresponds to a subset $\{i_1,i_2,\ldots,i_v\} \subseteq \mathbf{m}$. Then $2^{\mathbf{m}}\bigm\vert_{\Int(x)}$ is also Boolean, since it respects the relations given in $2^{\mathbf{m}}$. Hence, $\Int(x)$ is a Boolean lattice.
\end{proof}

\begin{prop}
\label{prop:mobiusvalues}
    For an arrangement $(\A,\T^n)$, we have $\mu(x) = (-1)^{\codim(x)}.$
\end{prop}

\begin{proof}
    We induct on the rank of an arbitrary element $x \in L(\A)$. As a base case, we take the minimal element $\T^n$ for which trivially $\mu(\T^n) = (-1)^{\codim(\T^n)} = 1$. For our induction hypothesis, we will assume that $\mu(x) = (-1)^{\codim{x}}$ for any $x \in L(\A)$ with rank less than $k$, for $k \leq n$. \footnote{Note that rank of an element $x \in L(\A)$ is equal to $\codim(x)$.} Suppose now that $\codim(x) = k$. By definition, $$\mu(x) = -\sum_{y \leq x} \mu(y) = \sum_{j < k} (-1)^{j+1}| \{ y \in \Int(x) \colon \codim(y) = j\}| .$$
    For each $j$, we have $| \{ y \in \Int(x) \colon \codim(y) = j\}| = \binom{k}{j}$, as the set can be thought of as a set of subsets with cardinality $j$. Hence $$\mu(x) = \sum_{j < k} (-1)^{j + 1} \binom{k}{j} = -(-1)^{k+1} = (-1)^k.$$
\end{proof}

\begin{defn}
    The \emph{characteristic polynomial} $\chi_\A(t)$ of the arrangement $(\A,\T^n)$ is defined by $$\chi_\A(t) = \sum_{x \in L(\A)} \mu(x) \cdot t^{\dim(x)}.$$
\end{defn}

It is clear from the definitions above that the coefficient of $t^{n-1}$ in $\chi_\A(t)$ is the number of hyperplanes of $\A$. We now state the toric version of Zaslavky's theorem.

\begin{prop}
    Let $(\A,\T^n)$ be an arrangement of hyperplanes. Then $$f_n (\A) = (-1)^{n}\chi_\A(0).$$
\end{prop}

\begin{proof}
    By \cref{prop:mobiusvalues}, for any element $x \in L(\A)$ with $\dim(x) = 0$ we have $\mu(x) = (-1)^n$. Therefore, the constant term in the characteristic polynomial $\chi_\A(t)$ is $(-1)^nf_0$. The result then follows by \cref{prop:facetsduality}.
\end{proof}

%%%%%%%%%%%%%%%%%%%%%%%%%%%%%%%% new section %%%%%%%%%%%%%%%%%%%%%%%%%%%%%%%%%%%%%%%%%%%%%%

\section{Classification of line arrangements}
\label{sec:sectionsix}

In the last section, we look into classifying hyperplane arrangements following from work in \cite{holmes2022affine}. For the remainder of the paper, we focus on arrangements in $\T^2$. We consider two arrangements to be combinatorially the same if one of them can be obtained from the other by continuously translating lines without ever creating a triple intersection point. 

\begin{defn}
\label{defn:equivrelation}
Let \((\mathcal{A}_1,\T^2)\) and \((\mathcal{A}_2,\T^2)\) be two arrangements. If it is possible to continuously translate hyperplanes of \(\mathcal{A}_1\) without creating a triple intersection to \(\mathcal{A}_2\), then \(\mathcal{A}_1 \sim \mathcal{A}_2\).
\end{defn}

\begin{exmp}
\label{exmp:classifyingarrangementsoftwolines}
Let \(\mathcal{A}\) be an arrangement on $\T^2$ consisting of lines given by the normal vectors \(\langle 1,1 \rangle, \langle -1, 1 \rangle,\) and \(\langle 1,0 \rangle\). We choose to fix \(\langle 1,1 \rangle, \langle -1, 1 \rangle\) and count by brute force the number of arrangements that arise by adding \(\langle 1,0 \rangle\). However, there is a restriction: we do not allow \(\langle 1,0 \rangle\) passing through points $A$ or $B$. This means that there are two regions where the line can fall into, hence there are at most two distinct arrangements in this case. However, we need to take the symmetry of $\T^2$ into consideration. Once we allow a global translation of the torus, the two arrangements are equivalent.

\begin{figure}[h!]
    \centering
    \begin{tikzpicture}[scale=.7]
        \node[outer sep=0,inner sep=0,minimum size=0] (v1) at (-5,5) {};
        \node[outer sep=0,inner sep=0,minimum size=0] (v2) at (0,5) {};
        \node[outer sep=0,inner sep=0,minimum size=0] (v4) at (0,0) {};
        \node[outer sep=0,inner sep=0,minimum size=0] (v3) at (-5,0) {};
        \node[outer sep=0,inner sep=0,minimum size=0] (v5) at (2,5) {};
        \node[outer sep=0,inner sep=0,minimum size=0] (v6) at (7,5) {};
        \node[outer sep=0,inner sep=0,minimum size=0] (v8) at (7,0) {};
        \node[outer sep=0,inner sep=0,minimum size=0] (v7) at (2,0) {};
        \fill[orange!10]  (-2.5,5) rectangle (v4);
        \fill[orange!10]  (v5) rectangle (4.5,0);
        \draw[dashed] (-2.5,5) node (v9) {} -- (-2.5,0);
	\draw[dashed] (4.5,5) -- (4.5,0);
        
        \draw [-latex] (v1) edge (v2);
        \draw [-latex] (v3) edge (v4);
        \draw [-latex] (v5) edge (v6);
        \draw [-latex] (v7) edge (v8);
        \draw [->>] (v3) edge (v1);
        \draw [->>] (v4) edge (v2);
        \draw [->>] (v7) edge (v5);
        \draw [->>] (v8) edge (v6);
        
        \draw[orange] (-1,0) -- (-1,5);
        \draw[red] (v1) -- (v4);
        \draw[blue] (v3) -- (v2);
        \draw[red] (v5) -- (v8);
        \draw[blue] (v7) -- (v6);
        \draw[orange] (3,5) -- (3,0);
        \node at (-3,2.5) {$A$};
        \node at (0.1,-0.3) {$B$};
        \node at (4.0,2.5) {$A$};
        \node at (1.8,-0.29) {$B$};
        \end{tikzpicture}
    \caption{An example of two arrangements consisting of the same hyperplanes which are equivalent under $\sim$.}
    \label{fig:distinctarrangements}
\end{figure}
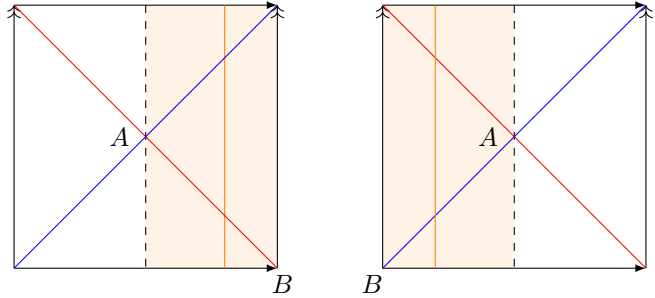
\end{exmp}

\begin{lem}
\label{lem:uniquearrangementsoftwolines}
    Let $(\A,\T^2)$ be an arrangement such that $\A=\{H_1,H_2\}$. Then there is a unique arrangement of hyperplanes $H_1$ and $H_2$.
\end{lem}

\begin{proof}
    It is not possible to have a triple intersection with only two lines. Therefore, we can translate the two lines so that their intersection is at the origin using a global translation of the torus.
\end{proof}

\begin{lem}
\label{lem:arrangementsofthreelines}
    Let $(\A,\T^2)$ be an arrangement of three hyperplanes $H_1,H_2,$ and $H_3$. Then there is a unique arrangement of the hyperplanes up to equivalence.
\end{lem}

\begin{proof}
    If a triple intersection occurs among the three hyperplanes, we can always fix it at the origin using global translation of the torus. In case that there are more than one triple intersection, we use the fact that the triple intersections must be equally spaced, therefore fixing any triple intersection at the origin gives the same arrangement under $\sim$.
\end{proof}

\subsection{An upper bound for the number of distinct line arrangements} In order to establish two arrangements are not equivalent, we make use of the intercept information that appears when defining a hyperplane. Translating a hyperplane is the same as changing the intercept. In consequence, we want to parametrise the intercept and describe a new arrangement, as we demonstrate in the next example.

\begin{exmp}
    We continue \cref{exmp:classifyingarrangementsoftwolines}. We write the hyperplanes in $\A$ as

    \begin{align*}
        &\theta_1: \T^2 \to S^1; \quad \theta(\Vec{x}) = x_1 + x_2 \\
        &\theta_2: \T^2 \to S^1; \quad \theta(\Vec{x}) = -x_1 + x_2 \\
        &\theta_3: \T^2 \to S^1; \quad \theta(\Vec{x}) = x_1
    \end{align*}
    
    Using the definition of hyperplanes, we describe the arrangement by
    \begin{align*}
    & \Theta : \T^2 \xrightarrow{(\theta_1,\;\theta_2,\;\theta_3)} (S^1)^3 \\
    &    \left( \begin{pmatrix}
            1 & 1 \\ -1 & 1 \\ 1 & 0
        \end{pmatrix},\begin{pmatrix}
            b_1 \\ b_2 \\ b_3
        \end{pmatrix} \right),
    \end{align*}
    for which we want to parameterise the intercepts $b_1, b_2,$ and $b_3$ corresponding to triple intersections of the hyperplanes. Therefore, we identify each arrangement with a point in $(S^1)^3$. Since the space of possible arrangements in going to be identified with $\T^3$, we look at each column of the original system of hyperplanes as a hyperplane in $\T^3$. We will call the first column $\Vec{a_1}$ and the second column $\Vec{a_2}$. We describe three arrangements in $(S^1)^3$ by taking the cross product of $\Vec{a_1}$ and $\Vec{a_2}$,
    \begin{equation*}
        \Vec{n} = \begin{pmatrix}
            1 \\ -1 \\1
        \end{pmatrix} \times \begin{pmatrix}
            1 \\ 1 \\ 0
        \end{pmatrix} = \begin{pmatrix}
            -1 \\ 1 \\ 2
        \end{pmatrix}.
    \end{equation*}
We claim that the hyperplane given by $\Vec{n}$ corresponds to the intercepts $b_1,b_2,$ and, $b_2$ giving a triple intersection. Observe that the $i$-th entry in $\Vec{n}$ gives the number of equivalence classes of arrangements we get by translating the $i$-th hyperplane, ignoring the symmetry of $\T^2$ for now. That is to say, taking the cross product gives us a vector whose entries are the solution to the restricted systems of linear equations. The $i$-th hyperplane is parameterized while the other two are fixed, so solving for the other two gives us the number of possible triple intersections when translating $H_i$. Since the $i$-th entry are hyperplanes on $S^1$ which correspond to line arrangements with a triple intersection, the 1-flats now represent an equivalence class of line arrangements under $\sim$, disregarding the symmetry of $\T^2$. For instance, the last entry suggests there are two distinct hyperplane arrangements under $\sim$, not taking the symmetry of $\T^2$ into consideration, which we have already intuitively verified.

\begin{figure}[h]
    \centering
    \begin{tikzpicture}[scale=.7]
          \tikzstyle{pt} = [fill,circle,outer sep=2,inner sep=2,minimum size=2]
          \tikzstyle{inv} = [outer sep=0,inner sep=0,minimum size=0]
            \node[outer sep=0,inner sep=0,minimum size=0] (v1) at (-3,5.5) {};
            \node[outer sep=0,inner sep=0,minimum size=0] (v2) at (2,5.5) {};
            \node[outer sep=0,inner sep=0,minimum size=0] (v4) at (2,0.5) {};
            \node[outer sep=0,inner sep=0,minimum size=0] (v3) at (-3,0.5) {};
            \node[outer sep=0,inner sep=0,minimum size=0] (v5) at (4,5.5) {};
            \node[outer sep=0,inner sep=0,minimum size=0] (v6) at (9,5.5) {};
            \node[outer sep=0,inner sep=0,minimum size=0] (v8) at (9,0.5) {};
            \node[outer sep=0,inner sep=0,minimum size=0] (v7) at (4,0.5) {};
            
            \draw [thin,-latex] (v1) edge (v2);
            \draw [thin,-latex] (v3) edge (v4);
            \draw [thin,-latex] (v5) edge (v6);
            \draw [thin,-latex] (v7) edge (v8);
            \draw [thin,->>] (v3) edge (v1);
            \draw [thin,->>] (v4) edge (v2);
            \draw [thin,->>] (v7) edge (v5);
            \draw [thin,->>] (v8) edge (v6);
            
            \draw[thick,orange] (-0.5,0.5) -- (-0.5,5.5);
            \draw[red] (v1) -- (v4);
            \draw[blue] (v3) -- (v2);
            \draw[red] (v5) -- (v8);
            \draw[blue] (v7) -- (v6);
            \draw[thick,orange] (4,5.5) -- (4,0.5);
    \draw  (3,-2.5) ellipse (6 and 1);
    \node[pt] at (0.4,-3.4) {};
    \node[pt] at (5.6,-1.6) {};
    
    \draw[decorate,decoration=zigzag,->] (-0.8,0.2) -- (0.2,-3);
    \draw[decorate,decoration=zigzag,->] (6.8,0.2) -- (5.8,-1.4);
    \end{tikzpicture}
    \caption{The intercept $b_3$ parametrised as an arrangement on $S^1$. The hyperplane $\langle 2 \rangle$ gives two arrangements with a triple intersection obtained by translating $H_3$.}
    \label{fig:exampleofparametrisinganintercept}
\end{figure}
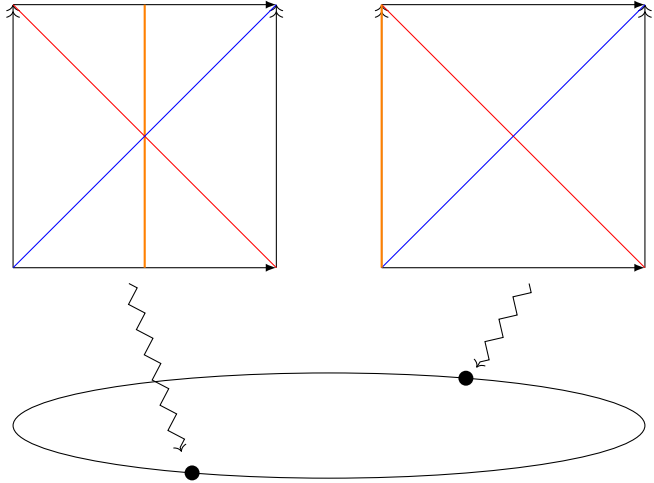

\end{exmp}

After the example, we go on to set the general construction. Considering an arrangement $(\A,\T^2)$ of $m$ lines, we put the normal vectors of hyperplanes into a system,
\begin{align*}
    (\Theta, \Vec{b}) = 
    \left(\begin{pmatrix}
        a_1 & a_2\\ a_3 & a_4 \\ a_5 & a_6 \\ \vdots & \vdots \\ a_{2m-1} & a_{2m}
    \end{pmatrix}, \begin{pmatrix} b_1\\b_2 \\ b_3 \\ \vdots \\ b_m \end{pmatrix}\right).
\end{align*}
Let $\mathcal{B}$ be a subarrangement of $\A$ consisting of three hyperplanes which are spanning (such a subarrangement always exists because $\A$ is spanning) for which we write
\begin{align*}
    \begin{pmatrix}
        a_i & a_{i+1}\\ a_{j} & a_{j+1} \\ a_k & a_{k+1}
    \end{pmatrix}\begin{pmatrix}
        x_1\\ x_2
    \end{pmatrix} = \begin{pmatrix}
        b_i\\b_j\\b_k
    \end{pmatrix}.
\end{align*}
We want to parametrise the intercepts $b_i,b_j,b_k$ of the hyperplanes and inspect which intercepts correspond to the subarrangement containing triple intersections. The arrangements with a triple intersection must be such that the system above has a solution.

\begin{lem}
    Given a system of hyperplanes $\Theta$ and an intercept vector $\Vec{b}$, there exists a vector $\Vec{x}$ satisfying
    \begin{align*}
        \begin{pmatrix}
            a_{11}&a_{12}\\a_{21}&a_{22}\\a_{31}&a_{32}
        \end{pmatrix} \begin{pmatrix} x_1\\ x_2\end{pmatrix} = \begin{pmatrix}b_1 \\ b_2 \\b_3\end{pmatrix}
    \end{align*}
    if and only if
    \[(\Vec{a_1} \times \Vec{a_2}) \cdot \Vec{b} = 0,\] where $\Vec{a_1}$ and $\Vec{a_2}$ stand for the first and second column of $\Theta$ respectively.
\end{lem}

\begin{proof}
    Suppose we have a vector $\Vec{x}$ which satisfies $\Theta \Vec{x} = \Vec{b}$. We rewrite the relation, 
    \begin{align*}
        \begin{pmatrix}
            a_{11} & a_{12} & b_1\\
            a_{21} & a_{22} & b_2\\
            a_{31} & a_{32} & b_3
        \end{pmatrix}\begin{pmatrix}
            x_1 \\ x_2 \\ -1
        \end{pmatrix} = \Vec{0}.
    \end{align*}
    Then, since there exists a non-zero vector in the kernel of $(\Vec{a_1}\mid \Vec{a_2} \mid \Vec{b})$, we have $\det(\Vec{a_1}\mid \Vec{a_2}\mid \Vec{b}) = 0$ which is equivalent to $(\Vec{a_1} \times \Vec{a_2})\cdot \Vec{b} = 0$ by Laplace expansion along the last column.
\end{proof}

Hence, $\Vec{n} \cdot \Vec{b} = (\Vec{a_1}\times \Vec{a_2}) \cdot \Vec{b}$ is the triple intersection locus in $(S^1)^3$ and $\Vec{n}$ defines a hyperplane parametrised by the intercepts. 

\begin{defn}
    Let $(\A,\T^2)$ be an arrangement of $m$ lines. Denote by $\Vec{n_{ijk}}$ the primitive cross product of column vectors of the system of hyperplanes for a spanning triple of normal vectors $\Vec{a_i},\Vec{a_j},\Vec{a_k}$. Let $\Vec{\tilde{n}_{ijk}}$ be the \emph{extended vector} of dimension $m$ where the $i$-th, $j$-th, and $k$-th coordinate are taken to be coordinates of $\Vec{n_{ijk}}$ respectively and the rest of the entries are zero.
\end{defn}

We define a new arrangement on $(S^1)^m$ of $\binom{m}{3}$ hyperplanes defined by normal vectors $\Vec{\tilde{n}_{ijk}}$ which is the parameter space of arrangements and is given by the system
\begin{align*}
    (\Tilde{\Theta}, \Vec{0}) = \left(
    \begin{pmatrix}
        \leftarrow \Vec{\Tilde{n}_{I_1}} \rightarrow\\
        \leftarrow \Vec{\Tilde{n}_{I_2}} \rightarrow\\
        \vdots\\
        \leftarrow \Vec{\Tilde{n}_{I_{\binom{m}{3}}}} \rightarrow
    \end{pmatrix},\; \Vec{0} \right),
\end{align*}
where the normal vectors of hyperplanes are indexed by unordered triples $I_i \subseteq \A$.

We can further refine the method. It follows from \cref{lem:uniquearrangementsoftwolines} that we can always fix the intercept of two lines at the origin in an arrangement $\A$. Therefore, for an arrangement of two hyperplanes $H_1 = \langle a_1,a_2 \rangle$ and $H_2= \langle a_3,a_4 \rangle$ on $\T^2$, we can fix the intercepts $b_1$ and $b_2$ of the two hyperplanes at the origin, that is, two entries of the intercept vector $\Vec{b}$ can be replaced by zero.

\begin{exmp}
Suppose we have hyperplanes $H_1,H_2,H_3,H_4$ on $\T^2$ given by normal vectors \[\Vec{a_1} = \langle 1,0 \rangle, \Vec{a_2} = \langle 0,1 \rangle, \Vec{a_3} = \langle 1,-2 \rangle, \Vec{a_4} = \langle 1,2 \rangle\] respectively. We fix the intercepts of $H_3$ and $H_4$ at the origin. Therefore, 
\begin{align*}
    (\Theta,\Tilde{\Vec{b}}) = \left(\begin{pmatrix}
        1&0\\ 0&1\\ 1&-2\\ 1&2
    \end{pmatrix}, \begin{pmatrix}
        b_1\\b_2\\0\\0
    \end{pmatrix}\right)
\end{align*}
is the system associated to the arrangement. We would like to construct a geometric space, say $\M$, where points are in correspondence with line arrangements and lines are in correspondence with arrangements containing triple intersections. Using the method above, we look at triples of spanning subarrangements of three hyperplanes. We start with the subarrangement with system of hyperplanes given by
\begin{align*}
    \left(\begin{pmatrix}
        1&0\\0&1\\1&-2
    \end{pmatrix}, \begin{pmatrix}
        b_1\\b_2\\0
    \end{pmatrix}\right),
\end{align*}
for which we compute the cross product of columns and then extend the normal vector,
\begin{align*}
    \Vec{n}_{123} = \begin{pmatrix}
        1\\0\\1
    \end{pmatrix}\times \begin{pmatrix}0\\1\\-2\end{pmatrix} = \begin{pmatrix}
        -1\\2\\1
    \end{pmatrix}, \qquad \Tilde{\Vec{n}}_{123} = \begin{pmatrix}
        -1\\2\\1\\0
    \end{pmatrix}.
\end{align*}
We repeat the process to obtain the extended system of hyperplanes in terms of $\Vec{b}$,
\begin{align*}
    (\Tilde{\Theta}, \Vec{0}) = \left( \begin{pmatrix}
        -1&2&1&0\\ -1&-2&0&1\\ 2&0&-1&-1\\ 0&4&1&-1
    \end{pmatrix}, \Vec{0} \right).
\end{align*}
Fixing the intercepts $b_3 = b_4 = 0$, we can reduce the arrangement to get
\begin{align*}
    (\Tilde{\Theta}_{\text{red}}, \Tilde{\Vec{b}}_{\text{red}}) = \left( \begin{pmatrix}
        1&-2\\1&2\\2&0\\0&4
    \end{pmatrix}, \Vec{0}\right).
\end{align*}
In addition, the intercepts of the hyperplanes are given by the fact that there is a unique arrangement where all four hyperplanes intersect, so each hyperplane has to pass through the origin. Since each hyperplane now corresponds to an arrangement containing a triple intersection, the $n$-flats represent distinct arrangements in general position.

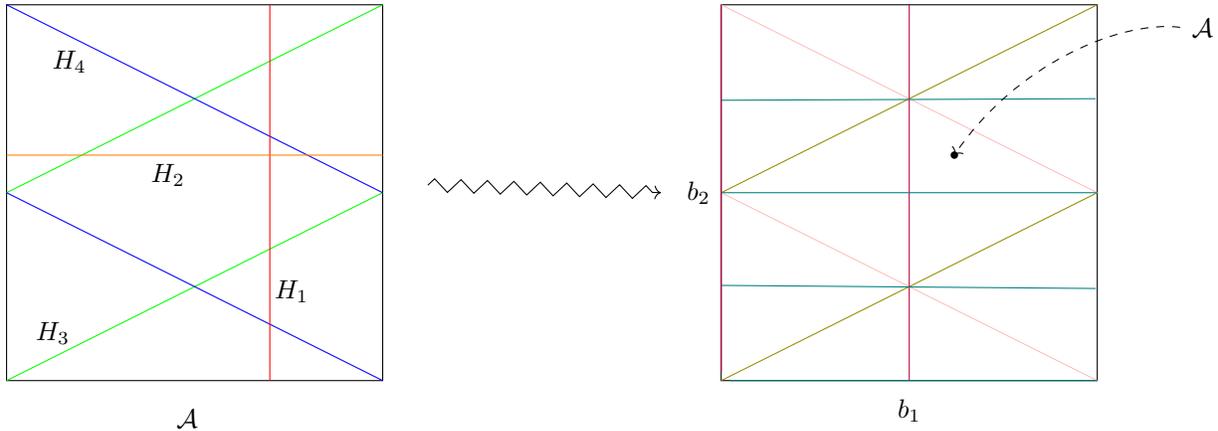
\begin{figure}[h!]
    \centering
    \usetikzlibrary{decorations.markings}
    \usetikzlibrary{decorations.pathmorphing}

\begin{tikzpicture}
        \draw  (-4,4.5) node[outer sep=0,inner sep=0,minimum size=0] (v4) {} rectangle (1,-0.5) node[outer sep=0,inner sep=0,minimum size=0] (v1) {};
        \draw [green](-4,-0.5) -- (1,2) node[outer sep=0,inner sep=0,minimum size=0] (v3) {};
        \draw [green](-4,2) node[outer sep=0,inner sep=0,minimum size=0] (v2) {} -- (1,4.5);
        \draw [blue](v1) -- (v2);
        \draw[blue](v3) -- (v4);
        \draw [red](-0.5,-0.5) -- (-0.5,4.5);
        \draw [orange](-4,2.5) -- (1,2.5);
        \node at (-3.3837,0.1314) {$H_3$};
        \node at (-3.165,3.7489) {$H_4$};
        \node at (-0.2083,0.6931) {$H_1$};
        \node at (-1.8543,2.2324) {$H_2$};
        \draw (5.5,4.5) node[outer sep=0,inner sep=0,minimum size=0] (v8) {} rectangle (10.5,-0.5) node[outer sep=0,inner sep=0,minimum size=0] (v5) {};
        \draw [olive](5.5,-0.5) node (v9) {} -- (10.5,2) node[outer sep=0,inner sep=0,minimum size=0] (v7) {} ;
        \draw[olive] (5.5,2) node[outer sep=0,inner sep=0,minimum size=0] (v6) {} -- (10.5,4.5);
        \draw [pink](v5) -- (v6);
        \draw[pink] (v7) -- (v8);
        \draw [purple](8,4.5) -- (8,-0.5);
        \draw [teal](5.4991,3.2297) -- (10.4774,3.2507);
        \draw[teal] (v6) -- (v7);
        \draw [teal](5.492,0.7668) -- (10.4818,0.7255);
        \node at (5.2,2) {$b_2$};
        \node at (8,-0.9) {$b_1$};
    	\draw [decorate,decoration=zigzag,->](1.6,2.1) -- (4.7,2);
	  \draw [purple] (v8) -- (v9);
        \draw [teal](v9) -- (v5);
\node at (-1.6,-1) {$\mathcal{A}$};
\node[fill,circle,minimum size=1,outer sep=1,inner sep=1] at (8.6,2.5) {};
\draw[dashed,->] (11.6,4.2) .. controls (10.8,4.3) and (9.4,3.7) .. (8.6,2.5);
\node at (11.9,4.2) {$\mathcal{A}$};
\end{tikzpicture}
    \caption{An example of constructing the space of arrangements consisting of hyperplanes $\M$ coming from vectors $\Vec{n_{123}}$ (olive), $\Vec{n_{124}}$ (pink), $\Vec{n_{134}}$ (purple), and $\Vec{n_{234}}$ (teal).}
    \label{fig:constructionfig}
\end{figure}
We can say more from the space of hyperplane arrangements. Since there are four arrangements with a quadruple intersection, the four points must be the same arrangement. Hence, the \cref{fig:constructionfig} suggests that the space of arrangements has a four-fold symmetry as a result of the additional symmetries of the torus. It suffices to check the fundamental domain of $\M$ to answer how many distinct hyperplane arrangements of $H_1,H_2,H_3,$ and $H_4$ we have, in this case there are 4 distinct arrangements. Note that if we fixed $b_1 = b_2 = 0$ instead, we would get an arrangement corresponding to the fundamental domain of the parametrized arrangement. This follows from the fact that fixing the hyperplanes with the least intersections gives us less additional symmetries.
\end{exmp}

The parameter space of arrangement is often times not in general position, therefore most of the results do not apply nicely to this arrangement. However, we can indicate an upper bound in the next theorem.

\begin{thm}
\label{thm:upperboundonarrangements}
    Let $(\A,\T^2)$ be a line arrangement of $m$ lines with $m > 3$ given by normal vectors $$\Vec{a_1},\Vec{a_2},\ldots,\Vec{a_m}.$$ Fix the intercept of two hyperplanes, given by $\Vec{a_g}$ and $\Vec{a_h}$ for which the absolute value of the determinant $\det(\Vec{a_g} \mid \Vec{a_h})$ is minimised, at the origin. We define an arrangement on $\T^{m-2}$ using the system $$\rho_{gh}(\Tilde{\Theta})\Vec{b} = \Vec{0},$$ where we get $\Tilde{\Theta}_{gh}$ by omitting the $g$-th and $h$-th column of $\Tilde{\Theta}$. Then the upper bound for the number of distinct arrangements of the $m$ hyperplanes is the number of $(m-2)$-flats of the arrangement described by $\rho_{gh}(\Tilde{\Theta})$,
    \[ \frac{1}{|\det(\Vec{a_g} \mid \Vec{a_h})|}\sum_{\substack{I_i \subset \{1,2,\ldots,m\}\\ |I|=3}} |\det \left( \rho_{gh}(\Vec{\Tilde{n}_{I_1}} \mid \Vec{\Tilde{n}_{I_2}} \mid \ldots \mid \Vec{\Tilde{n}_{I_m}}) \right)|. \]
\end{thm}

\begin{proof}
    Extending the normal vectors $\Vec{n_{ijk}}$ to $\Vec{\Tilde{n}_{ijk}}$ gives us a new arrangement on $\T^{m}$ where each hyperplane is the locus of $ijk$-intersection locus. By omitting two columns of the system of hyperplanes $\Tilde{\Theta}$, we remove some additional symmetry of the arrangement, and $\rho_{gh}(\Tilde{\Theta})$ now represents a spanning arrangement on $\T^{m-2}$. We can also definitely remove the symmetry of the fixed hyperplanes, which is the $|\det(\Vec{a_g} \mid \Vec{a_h})|$-fold symmetry. We count the $(m-2)$-flats using \cref{cor:regionsofspanningarrangements}, but note that the arrangement need not be in general position, hence we only get an upper bound.
\end{proof}

%%%%%%%%%%%%%%%%%%%%%%%%%%% bibliography %%%%%%%%%%%%%%%%%%%
\medskip

\bibliographystyle{alphaurl}
\bibliography{references}

\end{document}